\documentclass[a4paper,10pt]{amsart}
\usepackage[english]{babel}
\usepackage[utf8]{inputenc}
\usepackage[T1]{fontenc}

\usepackage{amssymb}
\usepackage{mathrsfs}
\usepackage{hyperref}
\usepackage{enumitem}	
\usepackage[dvipsnames]{xcolor}

\newcommand{\scr}[1]{\mathscr{#1}}
\newcommand{\frk}[1]{\mathfrak{#1}}
\newcommand{\bb}[1]{\mathbb{#1}}

\renewcommand{\rm}[1]{\mathrm{#1}}

\newcommand{\N}{\mathbb{N}}	
\newcommand{\Z}{\mathbb{Z}}	
\newcommand{\Q}{\mathbb{Q}}	
\newcommand{\C}{\mathbb{C}}	
\newcommand{\R}{\mathbb{R}}	
\newcommand{\Id}{\mathrm{Id}}	
\newcommand{\GL}{\mathrm{GL}}	
\newcommand{\Span}{\mathrm{span}}	
\newcommand{\dd}{\,\mathrm{d}}	
\newcommand{\de}{\partial}		

\renewcommand{\Re}{\mathrm{Re}}
\renewcommand{\Im}{\mathrm{Im}}
\newcommand{\THEN}{\Rightarrow}	

\theoremstyle{plain}
\newtheorem{proposition}{Proposition}[section]
\newtheorem{theorem}[proposition]{Theorem}
\newtheorem{lemma}[proposition]{Lemma}
\newtheorem{corollary}[proposition]{Corollary}
\newtheorem*{theorem*}{Theorem}
\newtheorem{thm}{Theorem}[section]

\theoremstyle{definition}
\newtheorem{definition}[proposition]{Definition}

\newcommand{\g}{\frk g}

\newcommand{\ad}{\rm{ad}}

\newcommand{\Der}{\mathtt{Der}}
\newcommand{\Aut}{\mathtt{Aut}}

\numberwithin{equation}{section}

\title{Metric Lie groups admitting dilations}

\author[Le Donne]{Enrico Le Donne}
\address[Le Donne]{Department of Mathematics and Statistics, University of Jyväskylä, Finland; and: Dipartimento di Matematica, Università di Pisa, Italy}
\email{ledonne@msri.org}

\author[Nicolussi Golo]{Sebastiano Nicolussi Golo}
\address[Nicolussi Golo]{Dipartimento di Matematica, Università di Padova, Italy}
\email{sebastiano2.72@gmail.com}

\thanks{E.L.D.~has been partially supported by the Academy of Finland (grant
288501
`\emph{Geometry of subRiemannian groups}')
and by the European Research Council
 (ERC Starting Grant 713998 GeoMeG `\emph{Geometry of Metric Groups}').
 S.N.G.~has been partially supported 
 	by the European Unions Seventh Framework Programme, Marie Curie Actions-Initial Training Network, under grant agreement n. 607643, ``Metric Analysis For Emergent Technologies (MAnET)'',
 	by the EPSRC Grant "Sub-Elliptic Harmonic Analysis" (EP/P002447/1),
	and by University of Padova STARS Project "Sub-Riemannian Geometry and Geometric Measure Theory Issues: Old and New".
	}
	

\date{\today}

\begin{document}
\begin{abstract}
		
We consider left-invariant distances $d$ on a Lie group $G$
with the property that there exists 
a multiplicative one-parameter group of Lie automorphisms
$(0,\infty)\to\Aut(G)$, $\lambda\mapsto\delta_\lambda$,
so that $ d(\delta_\lambda x,\delta_\lambda y) = \lambda d(x,y)$,
for all $x,y\in G$ and all $\lambda>0$.

First, we show that all such distances are admissible, that is, they induce the manifold topology.
Second, we characterize multiplicative one-parameter groups of Lie automorphisms that are dilations for some left-invariant distance in terms of algebraic properties of their infinitesimal generator.

Third, we show that an admissible left-invariant distance on a Lie group with at least one nontrivial dilating automorphism is biLipschitz equivalent to one that admits a one-parameter group of dilating automorphisms.
Moreover, the infinitesimal generator can be chosen to have spectrum in $[1,\infty)$.
Fourth, we characterize the automorphisms of a Lie group that are a dilating automorphisms for some admissible distance. 

Finally, we characterize metric Lie groups admitting a one-parameter group of dilating automorphisms as the only
locally compact, isometrically homogeneous metric spaces with metric dilations of all factors.
Such metric spaces appear as tangents of doubling metric spaces with unique tangents.

\end{abstract}
\maketitle

\setcounter{tocdepth}{1}
\phantomsection
\addcontentsline{toc}{section}{Contents}
\tableofcontents

\section{Introduction}

Lie groups endowed with a left-invariant distance
 that admits a metric dilation or a one-parameter family of metric dilations 
appear in several mathematical contexts.
Carnot groups 
 offer a non-commutative version of normed vector spaces and they appear as asymptotic cones of finitely generated groups with polynomial growth and as tangents of sub-Riemannian manifolds \cite{%
MR623534,
MR741395,
MR1421823,
MR1867362,
MR3267520,
MR3742567
}. 
Homogeneous groups are a further generalization.
They are simply connected metric Lie groups whose Lie algebra is graded, and they are endowed with a one-parameter family of diagonal dilations, that is, dilating automorphisms of the form $\delta_\lambda(x_1,x_2,\dots,x_n) = (\lambda^{w_1}x_1,\lambda^{w_2}x_2,\dots,\lambda^{w_n}x_n)$, see Example~\ref{sec12070916}.
Homogeneous groups appear in the study of PDE and singular integrals
\cite{MR657581,MR0442149,MR3469687}.

However, these cases don't exhaust all metric Lie groups admitting 
dilations.
There are indeed distances, already on the Abelian $\R^2$, that are not quasisymmetric to any of the homogeneous distances listed above, 
but they do admit a one-parameter family of dilations, see~\cite[Section~6]{MR2116315} and \cite{MR3180486}.
The additional complication is given by having dilations that can't be diagonalized, as in Example~\ref{ex:non-diagonal} below.
Following \cite{MR0353210}, these metric Lie groups appear as visual boundaries of homogeneous negatively curved manifolds, equipped with parabolic visual distances as introduced by Hamenstadt, see \cite{MR1016663, MR1476054}. 

\bigskip

Suppose that $G$ is a Lie group and $d$ is a left-invariant distance on $G$ that admits 
a multiplicative one-parameter group of Lie automorphisms
 $(0,\infty)\to\Aut(G)$, $\lambda\mapsto\delta_\lambda$
 so that
 \begin{equation}\label{eq12061454}
 d(\delta_\lambda x,\delta_\lambda y) = \lambda d(x,y)
 \qquad\forall x,y\in G,\ \forall \lambda>0 .
 \end{equation}
A multiplicative one-parameter group $(0,\infty)\to\Aut(G)$ 
is determined by a derivation $A\in\Der(\frk g)$ of the Lie algebra $\frk g$ of $G$ such that 
\begin{equation}\label{eq12061455}
(\delta_\lambda)_* = \lambda^A := e^{(\log\lambda)A} .
\end{equation}
Such $A$ is the infinitesimal generator of $\lambda\mapsto\delta_\lambda$ and 
we say that $d$ is \emph{$A$-homogeneous}.

If a left-invariant distance $d$ induces the manifold topology on $G$, then we say that $d$ is \emph{admissible} and that $(G,d)$ is a metric Lie group.
We don't require a priori that an $A$-homogeneous distance is admissible nor that $G$ is connected.
Instead, 
we prove in our first theorem that this is necessarily true.
Theorem~\ref{thm11171734} is proven in Section~\ref{sec11072237}.

\begin{thm}\label{thm11171734}
	Let $G$ be a Lie group with Lie algebra $\frk g$.
	Assume that $A\in\Der(\frk g)$ is so that $(0,\infty)\to\Aut(\frk g)$, $\lambda\mapsto\lambda^A$, defines a one-parameter group of Lie group automorphisms $\lambda\mapsto\delta_\lambda\in\Aut(G)$ with $(\delta_\lambda)_* = \lambda^A$.
	If an $A$-homogeneous distance on $G$ exists then it is  admissible and so  $G$ is connected.
\end{thm}

A derivation $A\in\Der(\frk g)$ induces a \emph{real grading} of $\frk g$, i.e., a splitting $\frk g = \bigoplus_{t\in\R}V_t$ with $[V_t,V_s]\subset V_{t+s}$, by means of the generalized eigenspaces of $A$.
In other words, after choosing a basis of $\frk g$ so that $A$ is in Jordan normal form, the blocks corresponding to eigenvalues with real part equal to $t$ determine the space $V_t$, see Proposition~\ref{prop12061739}.
Nonetheless, the derivation $A$ carries more structure 
 than just the grading, since $A$ may not be diagonalizable on $\R$, nor on $\C$. 

In our second result 
we characterize  when $A$-homogeneous distances exist.
Theorem~\ref{thm12041046} is proven in Section~\ref{sec12061717}.

\begin{thm}\label{thm12041046}
	Let $A$ be a derivation on the Lie algebra of a Lie group $G$
	with induced grading $\bigoplus_{t\in\R}V_t$.
	The following are equivalent:
	\begin{enumerate}[label=(\roman*)]
	\item\label{thm12041046item1}
	There exists an $A$-homogeneous distance on $G$;
	\item\label{thm12041046item2}
	The Lie group $G$ is connected and simply connected,
	each layer
	$V_t=\{0\}$ for all $t<1$ and
	the restriction $A|_{V_1}$ is diagonalizable over $\C$.
	\end{enumerate}
	In particular, if there exists an $A$-homogeneous distance on $G$, then $G$ is nilpotent.
\end{thm}

The implication from \ref{thm12041046item1} to \ref{thm12041046item2} is based on
known facts about contracting autormophisms (e.g., from \cite{MR812604}) and 
an example in $\R^2$ that was already present in \cite[Section~6]{MR2116315}, see also Examples~\ref{ex:non-diagonal} and~\ref{ex:without_distances} in this paper.
In the proof of \ref{thm12041046item2} implying \ref{thm12041046item1}, instead,
one needs to construct an $A$-homogeneous distance.
In the case $A$ is diagonalizable over $\R$,
this has been done already by Hebisch and Sikora, see \cite{MR1067309}.
Our construction is inspired by theirs.

Following \cite{MR812604}, we show that the presence of a single dilating automorphism already gives strong restrictions on the setting, as we next explain. 
A \emph{dilation of factor $\lambda$} of a metric space $(X,d)$ is a bijection $\delta:X\to X$ such that 
\[
d(\delta x,\delta y) = \lambda d(x,y)
\qquad \forall x,y\in X .
\]
We say that $\delta$ is \emph{nontrivial} if $\lambda\neq1$.
We recall from \cite{2017arXiv170509648C} that,
if $(X,d)$ is a metric Lie group and $\delta:(X,d)\to (X,d)$ is a nontrivial dilation,
then there is a unique simply connected nilpotent metric Lie group $(G,d)$ that is isometric to $(X,d)$.
Moreover, in this metric Lie group, and only in this Lie group structure, the dilation $\delta$ is a Lie group automorphism.
In this case, we call $(G,d,\delta,\lambda)$ a \emph{self-similar metric Lie group}.

Our third result explains the connection between homogeneous distances and self-similar distances.
Theorem~\ref{thm12061648} is proven in Section~\ref{sec12061813}.

\begin{thm}\label{thm12061648}
	If $(G,d,\delta,\lambda)$ is a self-similar metric Lie group,
	then there is $A\in\Der(\frk g)$ with eigenvalues belonging to $[1,\infty)$ and an $A$-homogeneous distance $d'$ on $G$ such that $\delta$ is also a dilation of factor $\lambda$ for $d'$.
	Moreover, for any such $A$ and $d'$, the identity map $(G,d)\to(G,d')$ is biLipschitz.
\end{thm}

Theorem~\ref{thm12061648} applies also to distances that are already $A$-homogeneous and states that, up to a biLipschitz change of the distance, we can assume the spectrum of $A$ to be real. This biLipschitz change is in fact necessary, see Proposition~\ref{prop12271522}.
We remark that one cannot reduce to the case when the derivation $A$ is diagonalizable.
For instance, the distances presented in Section~\ref{ex:non-diagonal} 
are not biLipschitz or even quasi-conformally equivalent to any homogeneous distance with diagonalizable dilating automorphisms (see also \cite[Section~6]{MR2116315} and \cite{MR3180486}).

Since the construction of the derivation $A$ in Theorem~\ref{thm12061648} is done by means of the Jordan block decomposition of $\delta$, one can reinterpret Theorem~\ref{thm12041046} in terms of the dilation $\delta$ as follows.
Theorem~\ref{thm12061651} is proven in Section~\ref{sec12062041}.

\begin{thm}\label{thm12061651}
	Let $G$ be a Lie group, $\delta\in\Aut(G)$ a Lie group automorphism and $\lambda\in(0,+\infty)\setminus\{1\}$.
	The following statements are equivalent
	\begin{enumerate}[label=(\roman*)]
	\item\label{thm12061651item1}
	There is an admissible distance on $G$ for which $\delta$ is a dilation of factor $\lambda$;
	\item\label{thm12061651item2}
	The Lie group $G$ is connected and simply connected,
	the eigenvalues of $\delta_*$ have modulus smaller than or equal to $\lambda$ if $\lambda<1$, 
	greater than or equal to $\lambda$ if $\lambda>1$,
	and the complexification of $\delta_*$ is diagonalizable on the generalized eigenspaces of the eigenvalues of modulus equal to $\lambda$.
	\end{enumerate}	
\end{thm}

Theorem~\ref{thm12061651} implies that any contracting automorphism of a connected Lie group is a dilation of a suitable factor $0<\lambda<1$ for some admissible distance.

If a distance admits a dilation of factor $\lambda$ for every $\lambda>0$, then we shall call it a \emph{homothetic distance}.
We recall from \cite{MR3646026} that isometries of nilpotent metric Lie groups are Lie group isomorphisms up to  left translations.
It follows that also metric dilations of nilpotent metric Lie groups are Lie group automorphisms up to  left translations.
Consequently, one can show that 
any homothetic admissible distance on a nilpotent Lie group $G$ is $A$-homogeneous for some derivation $A$, see Proposition~\ref{prop11071818}.
Together with Theorem~\ref{thm11171734}, this implies 
that
{\itshape
a left-invariant distance on a  nilpotent Lie group  is $A$-homogeneous for some derivation $A$ if and only if it is admissible and homothetic.}

This discussion allows us to prove a characterization in the spirit of the ones presented in~\cite{MR3283670} and in~\cite{2017arXiv170509648C}.
A statement similar in spirit to Theorem~\ref{thm12061855} can be found in the work of Buliga~\cite{MR2610783}.
Theorem~\ref{thm12061855} is proven in Section~\ref{buon_natale}.

\begin{thm}\label{thm12061855}
	If $X$ is a locally compact, isometrically homogeneous and homothetic metric space,
	then there 
	are a unique Lie group $G$, a derivation $A$ on its Lie algebra and 
	an $A$-homogeneous distance $d$ on $G$
	such that $(G,d)$  is isometric to $X$.
\end{thm}

We remark that in Theorems~\ref{thm12061648} and~\ref{thm12061855}
one cannot require in general that the spectrum of the derivation $A$ is real
without a biLipschitz modification of the distance. Indeed, 
we provide an $A$-homogeneous distance on $\R^2$ so that the only eigenvalue of $A$ is $2+i$ but $d$ is not $A'$-homogeneous for any $A'$ with real spectrum, see Example~\ref{sec12271414} and Proposition~\ref{prop12271522}.  
Some reduction to the real spectrum are possible in limited cases, see Proposition~\ref{prop12182017}.

We conclude with a theorem that readily follows using known results from \cite{MR2865538} and that is a generalization of \cite[Theorem 1.2]{MR2865538} to non-geodesic metric spaces.
\begin{thm}\label{thm05011958}
	Let $X$ be a metric space with a doubling measure $\mu$.
	Assume that $X$ has unique tangent at $\mu$-a.e. $p\in X$.
	Then, for $\mu$-a.e.~$p\in X$, the tangent $G_p$ of $X$ at $p$ is a Lie group endowed with a $A$-homogeneous distance, for some derivation $A$ of the Lie algera of $G_p$.
\end{thm}

\subsection*{Structure of the paper}
Section~\ref{sec10222252} contains several elementary facts that we need later.
Section~\ref{sec11072237} is devoted to the proof of Theorem~\ref{thm11171734}.
Section~\ref{sec12062042} presents basic properties of self-similar metric Lie groups and homothetic distances.
Section~\ref{sec12051423} contains examples of $A$-homogeneous distances and some pathological cases.
Section~\ref{sec12061717} is devoted to the proof of Theorem~\ref{thm12041046}.
Finally, Theorems~\ref{thm12061648},~\ref{thm12061651},~\ref{thm12061855} and~\ref{thm05011958}  are proven in Section~\ref{sec12061813}.

\subsection*{Acknowledgments}
This work has been prepared during two organized meetings titled ``Summer Holiday at Mum's Place'' in 2015 and 2018.
The authors wish to thank their mothers.

\section{Algebraic preliminaries}\label{sec10222252}

\subsection{Complexifications and generalized eigenspaces}

The \emph{complexification} of a finite-dimensional real vector space $V$ is the complex vector space $V_\C$ constructed as follows.
Define $V_\C=V\oplus V$ and $J:V_\C\to V_\C$ as $J(X,Y) := (-Y,X)$.
Then $V_\C$ becomes a complex vector space by defining $i\cdot(X,Y) := J(X,Y)$, 
where $i$ is the imaginary unit.
We identify elements $X\in V$ with $(X,0)\in V_\C$ and consequently $(X,Y) = X+JY$.
We also define the complex conjugate as $(X+JY)^* := X-JY$, whenever $X,Y\in V$.
Notice that $v\in V_\C$ belongs to $V$ if and only if $v^*=v$.

If $\phi:V\to V$ is a $\R$-linear map, then its \emph{complexification} is the $\C$-linear map $\phi:V_\C\to V_\C$, $\phi(X+JY) = \phi X + J\phi Y$.
The \emph{spectrum} of $\phi$ is defined by
\[
	\sigma(\phi): = \{\alpha\in\C: \det(\phi-\alpha\Id) = 0 \} 
\]
and the \emph{generalized eigenspace} of $\phi$ corresponding to $\alpha\in\C$ by
\[
	E^\phi_\alpha: = \{v\in V_\C:\exists n\in\N \quad (\phi-\alpha\Id)^nv = 0\} .
\]
We have that
$\phi E^\phi_0\subset E^\phi_0$,
$\phi E^\phi_\alpha = E^\phi_\alpha$ if $\alpha\neq0$
and
$V_\C = \bigoplus_{\alpha\in\sigma(\phi)} E^\phi_\alpha$.
Moreover, if $\psi$ is another linear map and $[\phi,\psi]=0$, then $\psi(E^\phi_\alpha)\subset E^\phi_\alpha$ for all $\alpha$.
In particular, one can split the space $V_\C = \bigoplus_{\alpha\in\sigma(\phi),\beta\in\sigma(\psi)} E^\phi_\alpha\cap E^\psi_\beta$, where each subspace $E^\phi_\alpha\cap E^\psi_\beta$ is preserved by both maps.

\begin{lemma}\label{lem08241611}
	If $A:V_\C\to V_\C$ is a $\C$-linear map on a complex vector space $V_\C$,
	then 
	$E^A_\alpha =
	E^{e^A}_{e^\alpha}$, for all $\alpha\in \C$.
\end{lemma}
\begin{proof}
	Fix $\alpha\in\C$.
	Since $AE^A_\alpha\subset E^A_\alpha$, then $e^AE^A_\alpha\subset E^A_\alpha$.
	If we show that $(e^A-e^\alpha\Id)|_{E^A_\alpha}$ is nilpotent, then we have ${E^A_\alpha}\subset E^{e^A}_{e^\alpha}$.
	Since $V_\C=\bigoplus_\alpha E^A_\alpha = \bigoplus_\alpha{E^{e^A}_{e^\alpha}}$, we can conclude $E^A_\alpha= E^{e^A}_{e^\alpha}$.
	
	So, without loss of generality, we assume $V_\C=E^A_\alpha$.
	For all $n\ge1$, define the polynomial $p_n(x,y) = x^{n-1}+x^{n-2}y+\dots+xy^{n-2}+y^{n-1}$, so that
	\(
	x^n-y^n = (x-y) p_n(x,y) 
	\).
	Let $m\in\N$ be such that $(A-\alpha\Id)^m=0$. 
	Then one can easily show
	\begin{align*}
	(e^A-e^\alpha\Id)^m
	&= \sum_{k_1,\dots,k_m=1}^\infty (A-\alpha\Id)^m \frac{p_{k_1}(A,\alpha\Id)}{k_1!}\cdots \frac{p_{k_m}(A,\alpha\Id)}{k_m!}
	= 0 .\qedhere
	\end{align*}
\end{proof}


If $\frk g$ is a real Lie algebra, we define Lie brackets on its complexification $\frk g_\C$ by
\[
[X_1+JY_1 , X_2+JY_2 ]_{\frk g_\C}
:= [X_1,X_2]-[Y_1,Y_2] +J \left(  [X_1,Y_2]+[Y_1,X_2] \right)  .
\]
With these Lie brackets, $\frk g_\C$ is a complex Lie algebra. 
We denote by $\Aut_\C(\frk g_\C)$ and $\Der_\C(\frk g_\C)$ the spaces of complex automorphisms and derivations of $\frk g_\C$, respectively.
The complexification of a Lie algebra automorphism of $\frk g$ is a Lie algebra automorphism of $\frk g_\C$.
Similarly, the complexification of a derivation is a derivation.
In other words, up to canonical identifications, $\Aut(\frk g)\subset\Aut_\C(\frk g_\C)$ and $\Der(\frk g)\subset\Der_\C(\frk g_\C)$.

\begin{lemma}\label{lemBurba}
	If $\phi\in\Aut_\C(\frk g_\C)$ and $\alpha,\beta\in\C$, then 
	\begin{equation*}
	[E^{\phi}_\alpha,E^{\phi}_\beta] \subset E^{\phi}_{\alpha\beta} , \qquad   \forall  \alpha,\beta\in\C .
	\end{equation*}
\end{lemma}
\begin{proof}
	The proof is elementary after one proves by induction on $n\in \N$ that
	\[
	(\phi-\alpha\beta\Id)^n[v,w] 
	= \sum_{\substack{j+k=n\\j,k\ge0 }} \binom{n}{j} [\alpha^k(\phi-\alpha\Id)^jv , \phi^j (\phi-\beta\Id)^k w] 
	\]
	holds for all $v,w\in\frk g_\C$, all $\alpha,\beta\in\C$ and all $n\in\N$.
	See also \cite[p.6, Prop.12]{MR0453824}. 
\end{proof}

\begin{lemma}\label{lem07251242}
	If $A\in\Der_\C(\frk g_{\C})$ and $\alpha,\beta\in\C$, then
	\begin{equation*}
	[E^A _\alpha,E^A_\beta]\subset E^A_{\alpha+\beta} .
	\end{equation*}
\end{lemma}
\begin{proof}
	The proof is elementary after one proves by induction on $n\in \N$ that
	\begin{equation*}
	\begin{split}
	(A-(\alpha+\beta)\Id)^n&[v,w] 
	= [(A-\alpha\Id)^nv,w]  \\
	&	+ 2\sum_{j=1}^{n-1} [(A-\alpha\Id)^j v,(A-\beta\Id)^{n-j} w] + [v,(A-\beta\Id)^n w] 
	\end{split}
	\end{equation*}
	holds for all $v,w\in\frk g_\C$, all $\alpha,\beta\in\C$ and all $n\in\N$.
\end{proof}
%

If $V_\C$ is a complex vector space, $L:V_\C\to V_\C$ is a linear map and $f:\sigma(L)\to\C$ is a function, we denote by $L_f$ the linear function such that $L_fv_\alpha=f(\alpha)v_\alpha$ for every $v_\alpha\in E^L_\alpha$.
One easily checks that $[L,L_f]=0$ and that, if $g:\sigma(L)\to\C$ is another map, then $[L_f,L_g]=0$.
Moreover, if $V_\C$ is the complexification of a real vector space $V$, $L(V)=V$ and $f(\bar\alpha)=\overline{f(\alpha)}$, then $L_f(V)=V$ again.

We will need the following two statements, whose easy proofs are based on Lemmas~\ref{lemBurba} and~\ref{lem07251242}.

\begin{lemma}\label{lem10271029}
	If $\phi\in\Aut_\C(\frk g_\C)$ and $f:\sigma(\phi)\to\C$ is a multiplicative function, i.e., $f(\alpha\beta)=f(\alpha)f(\beta)$ for all $\alpha,\beta\in\sigma(\phi)$, then $\phi_f\in\Aut_\C(\frk g_\C)$.
	
	In particular, if $\phi\in\Aut(\frk g)$ and $f:\sigma(\phi)\to\C$ is a multiplicative function with $f(\bar\alpha)=\overline{f(\alpha)}$, then $\phi_f\in\Aut(\frk g)$.
\end{lemma}

\begin{lemma}
	If $A\in\Der_\C(\frk g_\C)$ and $f:\sigma(A)\to\C$ is an additive function, i.e., $f(\alpha+\beta)=f(\alpha)+f(\beta)$ for all $\alpha,\beta\in\sigma(\phi)$, then $A_f\in\Der_\C(\frk g_\C)$.

	In particular, if $A\in\Der(\frk g)$ and $f:\sigma(A)\to\C$ is an additive function with $f(\bar\alpha)=\overline{f(\alpha)}$, then $A_f\in\Der(\frk g)$.
\end{lemma}

The following result is a straightforward consequence.

\begin{corollary}\label{cor07251444}
	If $A\in\Der(\frk g)$, then the linear maps $A_R, A_I, A_N:\frk g_{\C}\to\frk g_{\C}$ defined by
	\begin{align*}
	A_R(v) &= \Re(\alpha)v \qquad\text{ for }v\in E_\alpha,\ \alpha\in\C \\
	A_I(v) &= i\Im(\alpha)v \qquad\text{ for }v\in E_\alpha,\ \alpha\in\C \\
	A_N &= A-A_I-A_R .
	\end{align*}
	all belong to $\Der(\frk g)$ and they commute with one another and with $A$.
\end{corollary}

If $A\in\Der(\frk g)$  and $\lambda>0$, we denote by $\lambda^A$ the automorphism $e^{\log(\lambda)A}\in\Aut(\frk g)$.
Notice that $\lambda\mapsto\lambda^A$ is a group homomorphism $\R_{>0}\to\Aut(\frk g)$.
All one-parameter subgroups of $\Aut(\frk g)$ are of this form.

If $\frk g$ is the Lie algebra of the Lie group $G$, and if $A\in\Der(\frk g)$  and $\lambda>0$ are such that $\lambda^A$ induces a Lie group automorphism on $G$, then we will denote this Lie group automorphism again by $\lambda^A$.
This abuse of notation is safe when $G$ is connected simply connected, because every Lie algebra automorphism induces a unique Lie group automorphism of $G$.

\subsection{Gradings}\label{sec11051528}
In this paper we use the following terminology.
 A  {\em real grading} of a Lie algebra $\frk g$ is a family $(V_t)_{t\in \R}$ of linear subspaces of $\g$, where all but finitely many of the $V_t$'s are $\{0\}$, such that $\g$ is their direct sum
 \[
 \g=  \bigoplus_{t\in \R} V_t 
 \]
  and where 
  \[
  [V_t, V_u]\subset V_{t+u}, \qquad  \text{ for all } t,u >0.
  \]
  If there exists a real grading $(V_t)_{t\in \R}$ with $V_t=\{0\}$ for all $t\leq0$, then 
we say that a Lie algebra is \emph{positively graduable} and call
$(V_t)_{t\in (0,+\infty)}$ a  \emph{positive grading}  of $\g$. 
Every  positively graduable Lie algebra  is nilpotent.	

Both automorphisms and derivations of a Lie algebra define specific gradings, as we show in the following two propositions.

\begin{proposition}\label{prop08220937}
	Let $\phi\in\Aut(\frk g)$.
	For all $\lambda\in(0,+\infty)\setminus\{1\}$ and $t\in\R$, define 
	\[
	V_t 
	:= V_t(\lambda, \phi):= \frk g\cap \bigoplus_{|\alpha|=\lambda^t} E^\phi_\alpha .
	\]
	Then $\{V_t\}_{t\in\R}$ is a real grading of $\frk g$.
	Moreover,
	\[
	|\det(\phi)| 
	= \lambda^{\sum_{t\in\R} t\cdot\dim(V_t)}.
	\]
\end{proposition}
\begin{proof}
	For $\alpha\in\C$, define $U^\phi_\alpha := ( E^{\phi}_\alpha\oplus E^{\phi}_{\bar\alpha} )\cap\frk g$.
	We claim that
	\begin{equation}\label{eq12061733}
	\frk g = \bigoplus_{\alpha\in\sigma(\phi)} U^\phi_\alpha ,
	\end{equation}
	where the  sum is direct up to the identification $U^\phi_\alpha=U^\phi_{\bar \alpha}$.
	Indeed, let $v=\sum_{\alpha}v_\alpha\in\frk g$ with $v_\alpha\in E^\phi_\alpha$ for all $\alpha$.
	Notice that if $w\in E^\phi_\alpha$, then $w^*\in E^\phi_{\bar \alpha}$, because $(\phi-\alpha\Id)^nw = ((\phi-\bar\alpha\Id)^n w^*)^*$ for all $n\in\N$.
	Hence, since $v= v^*$, then $v_\alpha+v_{\bar \alpha} =  v_\alpha^* +  v_{\bar \alpha}^*$, where $ v_\alpha^*\in E^\phi_{\bar \alpha}$ and $ v_{\bar \alpha}^*\in E^\phi_\alpha$.
	Therefore, $ v_\alpha^* = v_{\bar \alpha}$, for all $\alpha$, and thus $v=\frac12 \sum_{\alpha}(v_\alpha+ v_\alpha^*)$, where $v_\alpha+ v_\alpha^*\in U_\alpha$.
		So, we have $\frk g= \sum_{\alpha\in\sigma(\phi)} U^\phi_\alpha$.
	Since $U^\phi_\alpha\cap U^\phi_\beta=\{0\}$ if $\alpha\notin\{\beta,\bar\beta\}$, the sum is direct.
	This proves claim~\eqref{eq12061733}.
	
	Since $\phi$ is injective, 
	then $U^\phi_0=\{0\}$.
	Therefore, by~\eqref{eq12061733}, we have $\frk g = \bigoplus_{t\in\R} V_t$.
	
	Using Lemma~\ref{lemBurba}, we have
	\begin{equation}\label{eq12061734}
	[U^\phi_\alpha,U^\phi_\beta] \subset U^\phi_{\alpha\beta}\oplus U^\phi_{\bar\alpha\beta} , \qquad \forall \alpha, \beta\in\C.
	\end{equation}
	
	If $X\in U^\phi_\alpha$ and $Y\in U^\phi_\beta$ with $|\alpha|=\lambda^t$ and $|\beta|=\lambda^s$, then $[X,Y]\in U^\phi_{\alpha\beta}\oplus U^\phi_{\bar\alpha\beta} \subset V_{t+s}$, because of~\eqref{eq12061734} and  $|\alpha\beta| = |\bar\alpha\beta| = \lambda^{s+t}$.
	Therefore, $[V_s,V_t] \subset V_{s+t}$ and $\{V_t\}_{t\in\R}$ is a real grading of $\frk g$.
	Finally, if we set $\varepsilon_\alpha=1$ if $\alpha\in\R$ and $\varepsilon_\alpha=1/2$ if $\alpha\in\C\setminus\R$,
	\[
	|\det(\phi)|
	= \left| \prod_{\alpha\in\sigma(\phi)} \alpha^{\dim_{\C}(E_\alpha)} \right| 
	= \prod_{\alpha\in\sigma(\phi)} |\alpha|^{\varepsilon_\alpha\dim_\R (U_\alpha)} 
	= \prod_{t\in\R} \lambda^{t \cdot\dim(V_t)} . \qedhere
	\]
\end{proof}

In a similar way, using Lemma~\ref{lem07251242} instead of Lemma~\ref{lemBurba}, one can show the following Proposition.
\begin{proposition}\label{prop12061739}
	Let $A\in\Der(\frk g)$ and define, for all $t\in\R$,
	\[
	V_t(A) := \frk g \cap \bigoplus_{s\in\R} E^{A}_{t+is} .
	\]
	Then $\{V_t(A)\}_{t\in\R}$ is a real grading of $\frk g$.
\end{proposition}

We finish this section by recalling a result due to Siebert \cite{MR812604} that we will use very often.
We provide the short proof for completeness.

\begin{proposition}[Siebert]\label{prop12131456}
	Let $G$ be a connected Lie group and $\delta\in\Aut(G)$ be a contractive automorphism, i.e., $\lim_{n\to\infty}\delta^nx=e_G$ uniformly on compact sets.
	Then $V_t=\{0\}$ for all $t\le0$ and $\lambda\in(0,1)$, where $V_t(\lambda, \delta_*)$ is as in Proposition~\ref{prop08220937},
	and
	$G$ is nilpotent and simply connected.
\end{proposition}
\begin{proof}
	Since $\delta$ is contractive, we have $\sigma(\delta_*)\subset\{|\alpha|<1\}$.
	Therefore, $V_t=\{0\}$ for all $t\le0$ and $\lambda\in(0,1)$.
	Since $\{ V_t \}_{t>0}$ is a positive grading and $G$ is connected, $G$ is nilpotent.
	Furthermore, $G$ is simply connected because otherwise, being $G$ nilpotent, there would be a nontrivial compact subgroup\footnote{Indeed, if $G$ is not simply connected, then $G=\tilde G/H$, where $\tilde G$ is simply connected and $H$ is a discrete central subgroup of $\tilde G$. If $h\in H\setminus\{e\}$, then $h=\exp(x)$ for some $x\in\frk g$, and $\exp(\R x)/(\exp(\R x)\cap H)$ is a torus in $G$.}
	$K$, and therefore $\{\delta^nK\}_{n\in\Z}$ would contain arbitrarily small subgroups of $G$. 
	Since $G$ is a Lie group, this cannot happen and therefore $G$ is simply connected.
\end{proof}

\section{The topology of $A$-homogeneous distances}\label{sec11072237} 
%

In the definition of $A$-homogeneous distance we gave in the Introduction, we don't require the distance to be admissible, i.e., to induce the manifold topology.
However, we prove that $A$-homogeneous distances are in fact admissible, as we stated in Theorem~\ref{thm11171734}.

This section is devoted to the proof of Theorem~\ref{thm11171734}, which consists of several steps.
In this section, $G$ is a Lie group with Lie algebra $\frk g$ and neutral element $e_G$,
$d$ is a left-invariant distance on $G$,
$A\in\Der(\frk g)$ is a derivation, 
$\lambda\mapsto\delta_\lambda\in\Aut(G)$ is a multiplicative one-parameter group of automorphisms such that $(\delta_\lambda)_*=\lambda^A$ and each $\delta_\lambda$ is a metric dilation for $d$ of factor $\lambda$, for each $\lambda>0$.
The topology of the Lie structure on $G$ and the one induced by $d$ are denoted by $\tau_G$ and $\tau_d$, respectively.
We denote by $G^\circ$ the $\tau_G$-connected component of $G$ containing $e_G$.
If $Z\subset G$ is a set and $\tau$ a topology on $G$, we use the convention $\tau\cap Z=\{U\cap Z:U\in\tau\}\subset 2^Z$.
If $V\in\tau\cap Z$, we will conventionally say that ``$V$ is $\tau$-open in $Z$'', even in case $V\notin\tau$.
We denote by $L_p$ the left translation by $p$ on $G$.

\subsection{First Step: Contractibility}
\begin{proposition}\label{prop08160956} 
	Every  eigenvalue  of $A$ has   strictly positive real part.
	Consequently, $\lim_{\lambda\to0}\delta_\lambda(p) = e$, uniformly on $\tau_G$-compact sets in~$G^\circ$.
\end{proposition}
\begin{proof}
	We choose a basis of $\g$ so that $A$ is in real-Jordan form: $A$ is a block diagonal matrix where each block is in one of the two forms
	\begin{equation}\label{eq08051712}
	J_a := 
	\begin{pmatrix}
	a & 1 & &  0\\
	 & \ddots & \ddots & \\
	 & & a & 1 \\
	0 & &  & a
	\end{pmatrix}
	\quad\text{or}\quad
	J_{ab} :=
	\begin{pmatrix}
	Z_{ab} & I & &  0\\
	 & \ddots & \ddots & \\
	 & & Z_{ab} & I \\
	0 & &  & Z_{ab}
	\end{pmatrix}
	,
	\end{equation}
	where $Z_{ab} := \begin{pmatrix} a & b \\ -b & a \end{pmatrix}$ and $I:=\begin{pmatrix} 1 & 0 \\ 0 & 1 \end{pmatrix}$, $a,b\in\R$ with $b\neq0$.
	
	We claim that in each block \eqref{eq08051712} the value $a$ is strictly positive.
	We need to consider five cases.
	
%
	
	\begin{enumerate}[leftmargin=*,label={(\arabic*)}]
	\item\label{item08051720} 	
	Consider $J_a$ with $a=0$.
	Then let $v\in\g$ be the vector that in each coordinate is zero except in the first one for $J_a$, where it is not zero.
	Hence $Av=av = 0$.
	Thus $\lambda^A v = \sum_{k=0}^\infty \frac{(\log(\lambda)A)^k}{k!} v = v$.
	Up to a scalar multiplication of $v$, we may suppose that $\exp(v)\neq e_G$.
	We reach a contradiction: for all $\lambda>0$ we have
	\[
	 	\lambda d(e,\exp(v)) 
		= d(e_G,\delta_{\lambda}\exp(v) ) 
		= d(e_G, \exp( \lambda^A v) ) 
		= d(e_G, \exp( v ) ),
	\]
	but the last term is a nonzero number independent on $\lambda$.
	\item\label{item08051739}
	Consider $J_a$ with $a<0$.
	Taking $v\in\g$ as in case \ref{item08051720}, we have $\lambda^Av = \sum_{k=0}^\infty \frac{(\log(\lambda)a)^k}{k!} v = \lambda^av $ and $\exp(v)\neq e_G$.
	Then we reach a contradiction: On the one hand, we have\footnotemark 
	\begin{align*}
	 	0\neq d(e_G,\exp(v))
		&\le d(e_G,\exp(\lambda^Av)) + d(\exp(\lambda^Av) , \exp(v)) \\
		&= d(e_G, \exp(\lambda^av) ) + d(\exp((\lambda^a-1)v) , e_G ) \\
		&= \lambda d(e_G,\exp(v)) + |\lambda^a-1|^{\frac1a} d(e_G,\exp(v)),
	\end{align*}
	that is \( \lambda-1 + |\lambda^a-1|^{\frac1a}\ge0 \) for all $\lambda>0$.
	On the other hand, if $a<0$, then $\lim_{\lambda\to0^+}\lambda-1 + |\lambda^a-1|^{\frac1a} = -1$.
	\footnotetext{
	Justification of the last equality: setting $\mu=(\lambda^a-1)^{1/a}$, we have 
	\[
	d(\exp((\lambda^a-1)v) , e_G )
	= d(\exp(\mu^a v),e) 
	= d(\exp(\mu^{A}v),e) 
	= \mu d(\exp(v),e) .
	\]
	}
	\item\label{item08051734}
	Consider the block $J_{ab}$ with $a=0$.
	Let $v_1$ (resp.~$v_2$) in $\g$ be the vector that in each coordinate is zero except in the first one (resp.~the second one) for $J_{ab}$. 
	Hence, for all $\lambda>0$,
	\begin{align*}
	 	\lambda^Av_1 &= \lambda^a (\cos(\log(\lambda)b)v_1 - \sin(\log(\lambda)b) v_2) \\
		\lambda^Av_2 &= \lambda^a (\sin(\log(\lambda)b)v_1 + \cos(\log(\lambda)b) v_2) .
	\end{align*}
	One may assume $\exp(v_1)\neq e_G$.
	Taking $\lambda_0=\exp(\frac{2\pi}{b})$ we reach a contradiction:
	\begin{align*}
	 	\lambda_0 d(e,\exp(v_1)) 
		= d(e, \exp(\lambda_0^A v_1))
		= d(e, \exp( v_1 ) ),
	\end{align*}
	since the last term is non zero, but $\lambda_0\neq1$.
	\item
	Consider the block $J_{ab}$ with $a<0$ and assume that, for $v_1$ and $v_2$ as in case \ref{item08051734}, $\Span\{v_1,v_2\}$ is not a commutative Lie algebra.
	Hence $v_3:=[v_1,v_2]\in\g\setminus\{0\}$.
	Then, since $\lambda^A$ is a Lie algebra automorphism, we get
	\begin{align*}
	 	\lambda^Av_3 
		= [\lambda^Av_1,\lambda^Av_2]
		= \lambda^{2a} (\cos^2(\log(\lambda)b) + \sin^2(\log(\lambda)b) ) [v_1,v_2]
		= \lambda^{2a} v_3 .
	\end{align*}
	By the argument in case \ref{item08051739}, we have a contradiction. 
	\item
	Consider a block $J_{ab}$ with $a<0$ and assume that, for $v_1,v_2\in\g$ as in case \ref{item08051734}, $\Span\{v_1,v_2\}$ is a commutative Lie algebra.
	Since $a<0$, the curve $\lambda\mapsto \lambda^{a+ib}\in\C$ is a spiral in the complex plane going to $\infty$ as $\lambda\to0^+$.
	Therefore, for all $N\in\N$ there are $\lambda_N,\mu_N\in(0,1/N)$ such that
	\begin{equation}\label{eq08051803}
	\lambda_N^{a+ib} + 1 = \mu_N^{a+ib} .
	\end{equation}
	One can geometrically prove the existence of such $\lambda_N$, $\mu_N$ by taking a point in the spiral at small parameter $\lambda$ 
	with horizontal tangent and then translating the horizontal line until the intersecting points differ by $(1,0)$.
	
	Notice that \eqref{eq08051803} implies
	\[
	\lambda_N^A + \Id = \mu_N^A
	\]
	in $\Span\{v_1,v_2\}$.
	Moreover, since $\Span\{v_1,v_2\}$ is an Abelian subalgebra, $\exp:\Span\{v_1,v_2\}\to G$ is a group morphism.
	
	Then we reach a contradiction:
	\begin{align*}
	 	0\neq d(e_G,\exp(v_1))
		&\le d(e_G, \exp(\mu_N^A v_1) ) + d(\exp(\mu_N^A v_1), \exp(v_1) ) \\
		&\le d(e_G, \exp(\mu_N^A v_1) ) + d(e_G, \exp((\mu_N^A-\Id)v_1)) \\
		&= d(e_G, \exp(\mu_N^A v_1) ) + d(e_G, \exp(\lambda_N^Av_1) ) \\
		&= (\mu_N + \lambda_N) d(e,v_1),
	\end{align*}
	because the last term tends to zero as $N\to\infty$.
	\end{enumerate}
	
	This completes the proof of our claim, i.e., that $a>0$ in each block~\eqref{eq08051712}.
	
	Recall that if
	\[
	J := 
	\begin{pmatrix}
	z & 1 & &  0\\
	 & \ddots & \ddots & \\
	 & & z & 1 \\
	0 & &  & z
	\end{pmatrix}
	\]
	is a $k\times k$ Jordan block with $z\in\C$, then
	\[
	e^{tJ} = 
	\begin{pmatrix}
	e^{tz} & \cdots & \frac{t^k}{k!} e^{tz} \\
	 & \ddots & \vdots \\
	0 & & e^{tz} 
	\end{pmatrix} ,
	\text{ that is, }
	\lambda^J = 
	\begin{pmatrix}
	\lambda^z & \cdots & \frac{\log(\lambda)^k}{k!} \lambda^z \\
	 & \ddots & \vdots \\
	0 & & \lambda^z 
	\end{pmatrix}  .
	\]
	Hence, if $\Re(z)>0$, then $\lambda^J \to 0$ as $\lambda\to0^+$.
	We deduce that for all $p\in\exp(\g)$
	\[
	\lim_{\lambda\to0^+} \delta_{\lambda}(p) = e_G .
	\]
	Let $p$ be in the connected component $G^\circ$ of the identity.
	Then there exist $p_1,\dots,p_m\in\exp(\g)$ such that $p=p_1\cdots p_m$.
	Therefore,
	\[
	\lim_{\lambda\to0^+} \delta_{\lambda}(p) 
	= \lim_{\lambda\to0^+} \delta_{\lambda}(p_1)\cdots\delta_{\lambda}(p_m) = e_G . \qedhere
	\]
\end{proof}
%

\subsection{Second Step: Proof of \texorpdfstring{$\tau_d\cap G^\circ\subset \tau_G\cap G^\circ$}{Tau d in Tau G}}
\begin{lemma}\label{lem08211414}
	There is $\Omega\subset G^\circ$ $\tau_G$-open such that $e_G\in\Omega$ and $\Omega\subset B(e_G,1)$.
\end{lemma}
\begin{proof}
	Let $v_1,\dots,v_n\in\frk g$ be a basis of $\frk g$.
	Define $\phi:\R^{n}\to G^\circ$ by
	\[
	\phi(t_1,\dots,t_{n}) = \bigodot_{j=1}^n (\exp(v_j) \cdot \delta_{e^{t_{j}}}\exp(v_j)^{-1}) .
	\]
	Set $\bar t=(0,\dots,0)$ and notice that $\phi(\bar t) = e_G$.
	Moreover,
	\begin{align*}
	\frac{\de \phi}{\de t_{j}} (\bar t)
	&= \left.\frac{\dd}{\dd t} \right|_{t=0} (\exp(v_j)\cdot \delta_{e^t} \exp(v_j)^{-1} )\\
	&= \dd L_{\exp(v_j)}|_{\exp(v_j)^{-1}} \left(\left.\frac{\dd}{\dd t} \right|_{t=0}\delta_{e^t} \exp(v_j)^{-1}  \right) .
	\end{align*}
	Recall (see for instance \cite[Thm~2.14.3]{MR746308}) that  the differential of the exponential map is
	\[
	\dd \exp|_x(y) 
	= \dd L_{\exp(x)}|_{e_G}\left( \sum_{k=0}^\infty \frac{(-1)^k}{(k+1)!} \ad_x^k(y) \right) , \qquad \forall x,y\in\frk g.
	\]
	For $w\in\frk g$, we have
	\begin{align*}
	\left.\frac{\dd}{\dd t}\right|_{t=0} \delta_{e^t} \exp(w) 
	&= \left.\frac{\dd}{\dd t}\right|_{t=0} \exp((\delta_{e^t})_*w) 
	= \dd\exp|_w \left.\frac{\dd}{\dd t} \right|_{t=0}e^{tA}w \\
	&= \dd L_{\exp(w)}|_{e_G} \left(\sum_{k=0}^\infty \frac{(-1)^k}{(k+1)!} \ad_w^k(Aw) \right) .
	\end{align*}
	Recall that $\exp(v_j)^{-1}=\exp(-v_j)$. 
	Using the latter formula with $w=-v_j$, we obtain
	\begin{align*}
	\frac{\de \phi}{\de t_{j}} (\bar t) 
	&= \dd L_{\exp(v_j)}|_{\exp(-v_j)} \left(-\dd L_{\exp(-v_j)}|_{e_G} \left(\sum_{k=0}^\infty \frac{1}{(k+1)!} \ad_{v_j}^k(Av_j) \right) \right) \\
	&= -\left(\sum_{k=0}^\infty \frac{1}{(k+1)!} \ad_{v_j}^k(Av_j) \right) .
	\end{align*}
	
	Since the real part of the eigenvalues of $A$ are strictly positive by Proposition~\ref{prop08160956}, we have that $\ker(A)=\{0\}$, i.e., that $(Av_1,\dots,Av_n)$ is a basis for~$\frk g$.
	
	We claim that $\left( \frac{\de \phi}{\de t_{j}} (\bar t) \right)_{j=1}^n$ is a basis for $\frk g$.
	Without loss of generality, we may assume that 
	$v_1,\dots,v_n\in\frk g$ is a basis that is adapted to the grading 
	induced by $A$ as in 
	Proposition~\ref{prop12061739}, which is positive because of Proposition~\ref{prop08160956}.
	We write this grading as $\frk g=\bigoplus_{\ell=1}^s V_{t_\ell}$ with $0<t_1<t_2<\dots<t_s$.
	Notice that, if $v_j\in V_{r_\ell}$, then $\ad^k_{v_j}(Av_j)\in V_{(k+1)r_\ell}$, 
	and thus we have
	\[
	\frac{\de \phi}{\de t_{j}} (\bar t) 
	= - Av_j \mod(V_{r_{\ell+1}}\oplus\dots\oplus V_{r_s}).
	\]
	One easily concludes that $\left( \frac{\de \phi}{\de t_{j}} (\bar t) \right)_{j=1}^n$ are linearly independent.
		
	We have obtained that $D \phi(\bar t)$ is surjective. 
	Hence, there is $\tilde\Omega\subset G^\circ$ open with $e_G\in\tilde\Omega$ and
	\[
	\tilde\Omega\subset \phi\{(t_j)_j: -1< t_j < 1\}  .
	\]
	So, if $p\in\tilde\Omega$, then $p=\phi(t_1,\dots,t_{n})$ with $t_j\in(-1,1)$.
	Therefore,
	\begin{align*}
	d(e_G,p)
	&\le \sum_{j=1}^n d(e_G,\exp(v_j)) + d(e_G,\delta_{e^{t_{j}}}\exp(-v_j)) \\
	&\le  \sum_{j=1}^n  d(e_G,\exp(v_j)) + e d(e_G,\exp(-v_j)) 
	< \infty
	\end{align*}
	Let $\epsilon =  \left(2 \sum_{j=1}^n [ d(e_G,\exp(v_j)) + e d(e_G,\exp(-v_j))] \right)^{-1}$ and define $\Omega=\delta_\epsilon\tilde\Omega$. 
	The proof is concluded because $\Omega$ is $\tau_G$-open, $e_G\in\Omega$ and $\Omega\subset B(e_G,1)$.
\end{proof}
\begin{proposition}\label{prop08161111}
	\[
	\tau_d\cap G^\circ \subset \tau_G \cap G^\circ .
	\]
\end{proposition}
\begin{proof}
	Let $U\in\tau_d\cap G^\circ$ and $p\in U$.
	Then there is $r>0$ such that $B(p,r)\cap G^\circ\subset U$.
	Therefore, if $\Omega$ is like in Lemma~\ref{lem08211414}, $p\delta_r\Omega\subset p\delta_rB(e_G,1)\cap G^\circ = B(p,r)\cap G^\circ\subset U$.
	Since $\Omega$ is $\tau_G$-open in $G^\circ$, then $p\in {\rm {int}}_{\tau_G\cap G^\circ}(U)$.
	Since this holds for all $p\in U$, then $U\in\tau_G\cap G^\circ$.
\end{proof}
\begin{corollary}\label{cor09042136}
	$d:G^\circ\times G^\circ\to[0,+\infty)$ is $\tau_G$-continuous.
\end{corollary}
\begin{proof}
	First, we prove that $p\mapsto d(e_G,p)$ is continuous in $e_G$.
	Indeed, if $p_k\overset{\tau_G}{\to} e_G$ in $G^\circ$ and $\epsilon>0$, then, by Proposition~\ref{prop08161111}, $\{p\in G^\circ:d(e_G,p)<\epsilon\}\in\tau_d\cap G^\circ\subset\tau_G\cap G^\circ$.
	Hence there is $N\in\N$ such that $d(e_G,p_k)<\epsilon$ for all $k>N$.
	
	Second, we prove that $d:G^\circ\times G^\circ\to[0,+\infty)$ is continuous.
	Let $p_k\overset{\tau_G}{\to} p$ in $G^\circ$ and $q_k\overset{\tau_G}{\to} q$ in $G^\circ$.
	Then
	\[
	|d(p_k,q_k)-d(p,q)| \le d(e_G,p_k^{-1}p) + d(e_G,q_k^{-1}q) \to 0. \qedhere
	\]
\end{proof}

\subsection{Third Step: Proof of \texorpdfstring{$\tau_G\cap G^\circ\subset \tau_d\cap G^\circ$}{Tau G in Tau d}}
\begin{lemma}\label{lem08161113}
	There is $\Omega\subset G^\circ$ $\tau_G$-precompact such that $B(e_G,1)\cap G^\circ\subset \Omega$.
\end{lemma}
\begin{proof}
	Let $\Omega_2\subset G^\circ$ be a $\tau_G$-open set in $G^\circ$ such that $e_G\in\Omega_2$ and $\bar\Omega_2$ is $\tau_G$-compact.
	Since, by Proposition~\ref{prop08160956}, $\lim_{t\to-\infty}\delta_{e^t}\Omega_2=\{e_G\}$ uniformly in $\tau_G$, there is a $\tau_G$-open set $\Omega_1\subset \Omega_2$ with $e_G\in \Omega_1$ such that $\delta_{e^t}\bar\Omega_1\subset \Omega_2$ for all $t\le0$.
	
	Since $\de\Omega_1$ is $\tau_G$-compact and does not contain $e_G$, and since $d$ is $\tau_G$-continuous on~$G^\circ$, then $m=\min\{d(e_G,p):p\in\de\Omega_1\}>0$.
	We claim that $B(e_G,1)\cap G^\circ\subset \delta_{1/m}\Omega_2\cup\Omega_2$.
	Let $p\in B(e_G,1)\cap G^\circ$.
	If $p\in\Omega_2$, then we are done.
	If $p\notin\Omega_2$, then there is $t<0$ such that $\delta_{e^t}p\in\de\Omega_1$, because $\lim_{t\to-\infty}\delta_{e^t}p=e_G$ by Proposition~\ref{prop08160956} and because the curve $t\mapsto \delta_{e^t}p$ is $\tau_G$-continuous.
	We have $m\le d(e_G,\delta_{e^t}p) = e^td(e_G,p) \le e^t$.
	Therefore, $\delta_mp = \delta_{me^{-t}}\delta_{e^t}p \in\delta_{me^{-t}}\bar\Omega_1 \subset\Omega_2$, because $me^{-t}\le1$.
	We conclude that $p\in \delta_{1/m}\Omega_2\cup\Omega_2$.
	
	The proof of the lemma is concluded, because $\Omega=\delta_{1/m}\Omega_2\cup\Omega_2$ is $\tau_G$-precompact.
\end{proof}
\begin{proposition}\label{prop08161112}
	\[
	\tau_G \cap G^\circ\subset \tau_d\cap G^\circ .
	\]
\end{proposition}
\begin{proof}
	Let $\Omega\in \tau_G\cap G^\circ$ as in Lemma~\ref{lem08161113}. 
	Since $\lim_{t\to0}\delta_t\Omega = \{e_G\}$  by Proposition~\ref{prop08160956}, the family of open sets $\{\delta_t\Omega\}_{t>0}$ is a system of $\tau_G$-neighborhoods of $e_G$ in $G^\circ$.
	If $U\in\tau_G\cap G^\circ$ and $p\in U$, then there is $r>0$ such that $p\delta_r\Omega\subset U$.
	Therefore, $B(p,r)\cap G^\circ = p\delta_r B(e_G,1)\cap G^\circ \subset p\delta_r\Omega\subset U$, i.e., $p\in {\rm {int}}_{\tau_d\cap G^\circ}(U)$.
	Since this holds for all $p\in U$, we obtain $U\in\tau_d\cap G^\circ$.
\end{proof}

\subsection{Fourth Step: \texorpdfstring{$G$}{G} is Connected}
\begin{lemma}\label{lem:1231}
	$G$ is connected. 
\end{lemma}
\begin{proof}
	Notice that, since both $\tau_d$ and $\tau_G$ are left-invariant and by Propositions~\ref{prop08161111} and~\ref{prop08161112}, we have for every $p\in G$
	\[
	\tau_G\cap pG^\circ = p(\tau_G\cap G^\circ) = p(\tau_d\cap G^\circ) = \tau_d\cap pG^\circ  .
	\]
	
	Let $p\in G$.
	Since the curve $t\mapsto \delta_{e^t}p$ is $\tau_G$-continuous, we have $\delta_{e^t}p\in pG^\circ$ for all $t\in\R$.
	Moreover, $d(p,\delta_{e^t}(p)) \le d(e_G,p) + e^td(e_G,p) < 2d(e_G,p)$ for all $t< 0$.
	Therefore, if $t<0$ then 
	\[
	\delta_{e^t}(p)\in B_d(p,2d(e_G,p))\cap p G^\circ . 
	\]
	We know from Lemma~\ref{lem08161113} that $B(p,r)\cap pG^\circ$ is $\tau_G$-precompact for every $r>0$.
	Therefore, there are $t_k\to-\infty$ and $q\in p\cdot G^\circ$ such that $\delta_{e^{t_k}}(p)\overset{\tau_G}{\to} q$.
	Since $\lim_{t\to-\infty} d(e,\delta_{e^t}(p)) = \lim_{t\to-\infty} e^td(e,p) = 0$ and since $d$ is $\tau_G$-continuous on $pG^\circ$ by Corollary~\ref{cor09042136}, we obtain $d(e_G,q)=0$, i.e., $p\in G^\circ$.
\end{proof}

\subsection{Conclusion of the proof of Theorem~\ref{thm11171734}}
By Lemma \ref{lem:1231} we have that 
	$G$ is connected, i.e., $G^\circ=G$. Hence, 
	Proposition~\ref{prop08161111} and Proposition~\ref{prop08161112} give $\tau_d= \tau_G$.
	\qed

\section{Homogeneous distances}\label{sec12062042}

\subsection{Self-similar Lie groups}
\begin{definition}
	A \emph{self-similar metric Lie group} is a quadruple $(G,d,\delta,\lambda)$ where $G$ is a Lie group, $d$ an admissible left-invariant distance on $G$, $\delta\in\Aut(G)$ and $\lambda\in(0,\infty)\setminus\{1\}$ so that
	\[
	d(\delta x,\delta y) = \lambda d(x,y) ,
	\qquad\forall x,y\in G .
	\]
\end{definition}

In Section~\ref{sec12051423} we present examples of $(G,d,\delta,\lambda)$ where $d$ is not admissible or $\delta$ is not a group automorphism.
In \cite{2017arXiv170509648C} it has been given a characterization of self-similar metric Lie groups:
\begin{theorem}[Cowling et al., \cite{2017arXiv170509648C}]\label{thm12041706}
	If a metric space is locally compact, connected, isometrically homogeneous, and it admits a metric dilation, then it is isometric to self-similar metric Lie group. Moreover, all metric dilations of a self-similar metric Lie group are automorphisms.
\end{theorem}

After a technical lemma about quotients of self-similar metric Lie groups, we show basic properties of self-similar metric Lie groups. 

\begin{lemma}\label{lem12041918}
	Let $(G,d,\delta,\lambda)$ be a connected self-similar metric Lie group and $H\lhd G$ a closed normal subgroup with $\delta(H)=H$.
	Then there are a left-invariant distance $\hat d$ on $\hat G:=G/H$ and an automorphism $\hat\delta\in\Aut(\hat G)$ such that $(\hat G,\hat d,\hat \delta,\lambda)$ is a self-similar metric Lie group and the quotient map $(G,d)\to (\hat G,\hat d)$ is a submetry.
\end{lemma}
\begin{proof}
	Since $\delta(H)=H$, then there is $\hat\delta\in\Aut(\hat G)$ with $\hat\delta(xH)=(\delta x)H$ for all $x\in G$.
	Define $\hat d:\hat G\times\hat G\to[0,+\infty)$ as 
	\[
	\hat d(xH,yH) := \inf\{d(xh,yk):h,k\in H\} .
	\]
	The function $\hat d$ is clearly symmetric, $G$-invariant and $\hat\delta$ rescales it by $\lambda$.
	Using the facts that $d$ is left-invariant, $H$ is a closed normal subgroup, and balls in $(G,d)$ are compact, one can show that $\hat d(xH,yH) = \min\{d(xh,y):h\in H\}$.
	It follows that $\hat d(xH,yH)=0$ if and only if $xH=yH$.
	Moreover, if $x,z\in G$, then $\hat d(xH,zH) = d(xh,z)$ for some $h\in H$;
	if $y\in G$, then $ d(xh,z)\le d(xh,yk)+d(yk,z)$ for all $k\in H$;
	taking the infimums on $k$ we obtain $$\inf_{k\in H} d(xh,yk)+d(yk,z) \le \inf_{k\in H} d(xh,yk)+ \inf_{k\in H}d(yk,z) \le \hat d(xH,yH) + \hat d(yH,zH).$$
	This shows that $\hat d$ is a distance on $\hat G$.
	Finally, since
	the quotient map $G\to\hat G$ is a submetry,
	$\hat d$ induces the manifold topology.
%
\end{proof}

\begin{theorem}[Structure of self-similar metric Lie groups]\label{teo08221641}
	Let $(G,d,\delta,\lambda)$ be a self-similar metric Lie group.
	Let $V_t:=V_t(\lambda, \delta_* )$ be 
	as in Proposition~\ref{prop08220937}. 
	The following facts hold:
	\begin{enumerate}[leftmargin=*,label=(\roman*)]
	\item\label{teo08221641item3}
	$G$ is connected, simply connected, and nilpotent;
	\item\label{teo08221641item1}
	$V_t=\{0\}$ for all $t<1$. In particular, 
	$ (V_t)_{t\geq 1} $ is a positive grading of $\g$;
	\item\label{teo08221641item4}
	 $(G,d)$ is a $Q$-Ahlfors regular metric   space with  
	\[
	Q := \sum_{t\geq 1} t\cdot\dim(V_t) .
	\]
	\end{enumerate}
	Suppose in addition that $\delta_* = \lambda^A$ for some $A\in\Der(\frk g)$.
	Then the following holds:
	\begin{enumerate}[resume*]
	\item\label{teo08221641item5}
	$E_{\lambda^\alpha}^{\delta_*} = E_{\alpha}^A$, for all  $\alpha\in\C$;
	\item\label{teo08221641item6}
	The eigenvalues of $A$ have all real part larger or equal than 1;
	\item\label{teo08221641item7}
	For $V_t$ as above and $V_t(A)$ as in Proposition~\ref{prop12061739}, we have $V_t = V_t(A)$.
	\end{enumerate}
\end{theorem}
\begin{proof}
	Since $G$ is locally connected and $d$ is admissible, there are a connected neighborhood $U$ of $e_G$ and a radius $r>0$ such that $B_d(e_G,r)\subset U$.
	Since $ B_d(e_g,\lambda^nr)= \delta^nB_d(e_G,r)\subset \delta^nU $ for all $n\in\Z$ and since $G=\bigcup_{n\in\Z}B_d(e_g,\lambda^nr)$, 
	then $G=\bigcup_{n\in\Z}\delta^nU$.
	Since $e_G\in\bigcap_{n\in\Z}\delta^nU$, then $G$ is connected.
	
	Since $d$ is admissible and $\lambda\neq1$,
	either $\delta$ or $\delta^{-1}$ is a contractive automorphism of $G$.
	More precisely,
	if $\lambda<1$, then $\delta$ is contractive; 
	if $\lambda>1$, then $\delta^{-1}$ is contractive and $V_t(\lambda,\delta_*) = V_t(1/\lambda,\delta_*^{-1})$.
	Using Proposition~\ref{prop12131456}, we've got in both cases that $V_t=\{0\}$ for all $t\le0$, and that $G$ is simply connected and nilpotent.
	Item~\ref{teo08221641item3} is thus proven.
	
	 	
	Set $Q:= \sum_{t>0} t\cdot\dim(V_t)$. Notice that we have not proved yet that 
	the elements in the last sum  are zero for $t<1$. 
	If $\mu$ is a Haar measure on $G$, then, for all $n\in\Z$,
	\[
	\mu(B(e_G,\lambda^n)) 
	= \mu(\delta^n(B(e_G,1)))
	= |\det\delta_*|^n\cdot\mu(B(e_G,1))
	= \lambda^{nQ} \cdot \mu(B(e_G,1)) ,
	\]
	using Proposition~\ref{prop08220937} in the last identity.
	It follows that $(G,d)$ is Ahlfors regular with Hausdorff dimension~$Q$, see~\cite[\S8.7]{MR1800917}.
	Therefore, the point~\ref{teo08221641item4} is proven (without item~\ref{teo08221641item1}).
	
	We want to show that $t_m:=\min\{t\in\R:V_t\neq\{0\}\}$ is larger or equal than 1.
	Define $\frk h=\bigoplus_{t>t_m}V_t$.
	We claim that $[\frk g,\frk g]\subset\frk h$.
	Indeed, if $X\in V_{t}$ and $Y\in V_s$ with $s,t \geq t_m$, then $[X,Y]\in V_{t+s}$, since $t+s\geq 2 t_m>t_m$, because $t_m>0$.
	
	Let $H=\exp(\frk h)<G$ be the connected Lie subgroup associated to $\frk h$ and define $\hat G=G/H$.
	Notice that $H$ is a closed normal subgroup and that $\delta(H)=H$, because $G$ is simply connected and nilpotent, $[\frk g,\frk h]\subset[\frk g,\frk g]\subset\frk h$ and $\delta_*(\frk h)=\frk h$.
	By Lemma~\ref{lem12041918}, there is a distance $\hat d$ on $\hat G$ so that the quotient map $(G,d)\mapsto(\hat G,\hat d)$ is a submetry and  $(\hat G,\hat d,\hat \delta,\lambda)$ is a self-similar metric Lie group.
	By the item~\ref{teo08221641item4}, which we proved above, we have
	\[
	\dim(V_{t_m}) =
	\dim_{\rm {top}}\hat G  
	\le \dim_{\rm {Haus}}(\hat G,\hat d)
	= t_m \dim(V_{t_m}),
	\]
	and therefore $t_m\ge1$.
	This completes the proofs of item  \ref{teo08221641item1}. 
	
	We  consider now the last items of the theorem: assume that $\delta_*=\lambda^A$, for some $A\in\Der(\frk g)$. 
	Notice that for all $\beta,k\in\C$ we have $E^{kA}_{k\beta} = E^A_{ \beta}$, because $(kA-k \beta\Id)^n = k^n(A- \beta\Id)^n$.
	Therefore, using Lemma~\ref{lem08241611},  
	 we have 
	$E^{\delta_*}_{\lambda^\alpha} 
	= E^{(\log\lambda)A}_{(\log \lambda)  \alpha}
	= E_{\alpha}^A$.
	Items~\ref{teo08221641item5},~\ref{teo08221641item6} and~\ref{teo08221641item7} readily follow.
\end{proof}

\subsection{Homothetic self-similar metric Lie groups}
A distance $d$ on a Lie group $G$ is said \emph{homothetic} if it is left-invariant, admissible and it admits a metric dilation of factor~$\lambda$ for every $\lambda>0$.

Any $A$-homogeneous distance on $G$ is clearly homothetic.
Since homothetic distances are assumed to be admissible, from Theorems~\ref{thm12041706} and~\ref{teo08221641} we obtain that, up to possibly changing the group structure, we may assume $G$ to be nilpotent.
In this case, one easily shows that homothetic distances are $A$-homogeneous, for some derivation $A$:
\begin{proposition}\label{prop11071818}
	Let $d$ be a homothetic distance on a nilpotent Lie group $G$.
	Denote by $P$ the group of dilations of $(G,d)$ and by $I$ the subgroup of $P$ consisting of isometries, i.e., dilations of factor 1.
	
	Then $P=I\rtimes\R_{>0}$ and it is a closed subgroup of $G\rtimes\Aut(G)$.
	In particular, there is $A\in\Der(\frk g)$ so that $d$ is an $A$-homogeneous distance.
\end{proposition}
\begin{proof}
	Notice that a dilation $(G,d)\to(G,d)$ of factor $\lambda$ is  an isometry $(G,\lambda d)\to(G,d)$.
	Since isometries of nilpotent Lie groups are affine maps, see \cite{MR3646026}, then $P\subset G\rtimes\Aut(G)$.
	
	Now, we claim that in fact $P$ is a closed subgroup: indeed, if $\{\delta_n\}_{n\in\N}$ is a sequence of dilations, each of factor $\lambda_n>0$, converging to $\delta$ in $G\ltimes\Aut(G)$, then it converges pointwise.
	Therefore, if $x,y\in G$ are such that $d(x,y)\neq0$, then
	\[
	\lim_{n\to\infty} \lambda_n 
	= \lim_{n\to\infty} \frac{d(\delta_n x,\delta_n y)}{d(x,y)} 
	= \frac{d(\delta x,\delta y)}{d(x,y)} .
	\]
	Since $\delta$ is still bijective, it being in $G\ltimes\Aut(G)$, then $\lambda:=\frac{d(\delta x,\delta y)}{d(x,y)}>0$.
	Since the limit $\lim_{n} \lambda_n$ does not depend on the choice of $x$ and $y$, $\delta$ is a dilation of factor $\lambda$.
	
	Since $P$ is a closed subgroup of $G\ltimes\Aut(G)$, then it is a Lie group.
	Notice that the map $f:P\to\R_{>0}$ that associates to each dilation its dilation factor is a continuous surjective group morphism.
	Notice that if $\delta$ is a nontrivial dilation of $(G,d)$, then $(\delta e_G)^{-1}\delta$ is a  dilation fixing $e_G$, thus an automorphism of $G$, and with the same factor as $\delta$.
	Therefore, the restriction $f:P\cap\Aut(G)\to \R_{>0}$ is still surjective.
	In particular, there is a one-parameter subgroup $S\subset P\cap\Aut(G)$ such that the restriction $f|_S:S\to \R_{>0}$ is an isomorphism.
	Since $I=\ker(f)$, then $I$ is a closed normal subgroup of $P$ and $P=I\rtimes S$.
	Finally, $d$ is $A$-homogeneous for some infinitesimal generator $A\in\Der(\frk g)$ of $S$.
\end{proof}

Notice that, if there exists an admissible $A$-homogeneous distance on $G$, then Theorem~\ref{teo08221641} applies.
In particular, $G$ is a connected, simply connected, nilpotent Lie group.
In the rest of this section we will only prove some technical results that we will need later in Sections~\ref{sec12061717} and \ref{sec12061813}.

\begin{lemma}\label{lem10011902}
	Suppose that $G$ is a connected, simply connected, nilpotent Lie group with Lie algebra $\frk g$.
	Let $A\in\Der(\frk g)$. 
	A set $B\subset G$ is the closed unit ball of an $A$-homogeneous distance if and only if
	\begin{enumerate}[label=(\roman*)]
	\item\label{lem10011902item1}
	$e_G\in\rm{int}(B)$ and $B$ is compact;
	\item\label{lem10011902item2}
	$B^{-1}=B$;
	\item\label{lem10011902item3}
	$B$ is $A$-convex, i.e., for all $x,y\in B$ and all $\lambda\in[0,1]$
	\[
	(\lambda^A x)\,((1-\lambda)^Ay) \in B ,
	\]
	where we use the convention $0^A\equiv e_G$.
	\end{enumerate}
\end{lemma}
\begin{proof}
	The fact that \ref{lem10011902item1}--\ref{lem10011902item3} follow from $B$ being the unit ball of an $A$-homogeneous distance is straightforward.
	We shall prove the converse implication.
	
	Define $d(p,q):=N(p^{-1}q)$, with $N(p) := \inf\{\mu > 0 \,:\, \mu^{-A} p \in B\}$.
	We shall prove that $d$ is an $A$-homogeneous distance and $B = \{p\, :\, d(e_G,p) \le 1\}$.
	
	Clearly $d\ge0$, $d$ is symmetric and left-invariant, and $d(\lambda^Ax,\lambda^Ay)=\lambda d(x,y)$.
	By the continuity of the action $\lambda\mapsto\lambda^A$ and by the compactness of $B$, we have $B = \{p\, :\, d(e_G,p) \le 1\}$.
	
	From \ref{lem10011902item3} and the facts that $e_G\in B$ and $\lambda^Ae_G=e_G$, we have:
	\begin{equation}\label{eq12051028}
	\begin{array}{c}
	\text{for all $n\in\N$, $x_1,\dots,x_n\in B$ and $\lambda_1,\dots,\lambda_n\in[0,1]$ with $\sum_j\lambda_j\le1$,}\\
	(\lambda_1^Ax_1)\cdots (\lambda_n^Ax_n) \in B .
	\end{array}
	\end{equation}
	The proof of~\eqref{eq12051028} proceeds by induction on $n$.
	If $n=1$, then it follows from \ref{lem10011902item3} with $y=e_G$.
	If \eqref{eq12051028} holds up to $n$, one can prove it for $n+1$ using the fact that
	\begin{multline*}
	(\lambda_1^Ax_1)\cdots (\lambda_{n-1}^Ax_{n-1}) (\lambda_n^Ax_n)(\lambda_{n+1}^Ax_{n+1}) \\
	= (\lambda_1^Ax_1)\cdots (\lambda_{n-1}^Ax_{n-1})(\lambda_n+\lambda_{n+1})^A \left[\left(\left(\frac{\lambda_n}{\lambda_n+\lambda_{n+1}}\right)^Ax_n\right)\left(\left(\frac{\lambda_{n+1}}{\lambda_n+\lambda_{n+1}}\right)^Ax_{n+1}\right)\right] .
	\end{multline*}
	
	We claim that $d<\infty$, that is, for every $p\in G$ there is $\lambda>0$ such that $\lambda^Ap\in B$.
	Fix $p\in G$.
	Since $B$ is an open neighborhood of $e_G$ and $G$ is connected, then there is $n\in\N$ and $x_1,\dots,x_n\in B$ such that $p=x_1\cdots x_n$.
	By \eqref{eq12051028}, we have $(1/n)^Ap=(1/n)^Ax_1\cdots (1/n)^Ax_n\in B$.
	The claim is proven.
	
	We claim that $N(p)=0$ implies $p=e_G$.
	Indeed, if $N(p)=0$ then there is a sequence $\mu_n\to0$ with $\left(\frac{1}{\mu_n}\right)^{A}p\in B$.
	We can suppose that $\mu_n\le 1/n$.
	Therefore, from~\eqref{eq12051028} we deduce that, for all $n\in\N$,
	\[
	p^n = \left(\mu_n^A(\mu_n^{-A}p)\right)^n \in B .
	\]
	Similarly, since $B=B^{-1}$, then $N(p^{-1})=0$ and so $p^{-n}\in B$ for all $n\in\N$.
	It follows that the closed group $\overline{\{p^{n}: n\in\Z\}}$ is contained in $B$ and thus is a compact subgroup of $G$.
	Since $G$ is simply connected and nilpotent, the only compact subgroup is $\{e_G\}$ and thus $p=e_G$.
	
	The triangle inequality $N(xy)\le N(x)+N(y)$ follows from the $A$-convexity of~$B$: 
	If we set $a=N(x)$ and $b=N(y)$ and they are both nonzero, then $A$-convexity of $B$ implies
	\[
	\frac{N(xy)}{a+b} 
	= N\left( \left(\frac{a}{a+b}\right)^{A} a^{-A}x\, *\, \left(\frac{b}{a+b}\right)^{A} b^{-A}y \right)
	\le 1 .
	\]
	We conclude that $d$ is an $A$-homogeneous distance on $G$.
\end{proof}

For the proof of the following lemma, see \cite{2017arXiv170509648C}\footnote{In the first arXiv version of \cite{2017arXiv170509648C} it was  Lemma~3.3.}.
\begin{lemma}\label{lem10290953}
	Let $d$ be an admissible distance of a Lie group $G$ and $\scr K\subset\Aut(G)$ a compact group of automorphisms.
	Then the distance
	\[
	d'(x,y) := \max\{d(kx,ky):k\in\scr K\}
	\]
	is an admissible distance on $G$ and it is $\scr K$-invariant.
	Moreover, if $\delta$ is a metric dilation of factor $\lambda$ for $d$ that commutes with $\scr K$, i.e., $\delta\scr K\delta^{-1}=\scr K$, then it is also a dilation of factor $\lambda$ for $d'$.
\end{lemma}

\begin{lemma}\label{lem10291010}
	If $d_1$ and $d_2$ are two admissible distances on a Lie group $G$ and $\delta\in\Aut(G)$ 
	is a dilation of factor $\lambda\neq1$ for both distances,
	then the identity map $(G,d_1)\to (G,d_2)$ is biLipschitz.
\end{lemma}
\begin{proof}
	We need to show that there are $L_1,L_2>0$ such that, for all $x\in G$,
	\[
	L_1 d_1(e_G,x) \le d_2(e_G,x) \le L_2 d_1(e_G,x) .
	\]
	We will show only the second one, because then the first one follows by exchanging the roles of $d_1$ and $d_2$.
	Without loss of generality, we can assume $\lambda>1$.
	Let $B_j$ be the ball centered at $e_G$ of radius 1 with respect to $d_j$.
	Then there is $k\in\Z$ such that $\delta^kB_1\subset B_2$.
	Let $x\in G\setminus\{0\}$.
	There is $\ell\in\Z$ such that $x\in\delta^{\ell+1} B_1$ but $x\notin\delta^{\ell}B_1$, i.e., $\lambda^\ell \le
	d_1(e_G,x)
	\le \lambda^{\ell+1}$.
	Therefore,
	\[
	d_2(e_G,x)
	= \lambda^{\ell+1-k} d_2(e_G,\delta^k\delta^{-(\ell+1)}x) 
	\le \lambda^{1-k} \lambda^\ell 
	\le \lambda^{1-k} d_1(e_G,x) .\qedhere
	\]
\end{proof}

\section{Examples}\label{sec12051423}

For the first three examples, we consider $\R^2$  as Abelian Lie group.

\subsection{Some trivial examples}\label{sec12070916}
If $\alpha,\beta\ge1$, the (diagonalizable) matrix 
$A= \begin{pmatrix} \alpha & 0 \\ 0 & \beta \end{pmatrix}$ gives rise to
automorphisms $\delta_\lambda :=  \lambda^A= \begin{pmatrix} \lambda^\alpha & 0 \\ 0 & \lambda^\beta \end{pmatrix}$.
These maps are  one-parameter groups of dilating automorphisms  
  for several distances such as $d((x,y),(x',y')) = \max\{|x-x'|^{1/\alpha},|y-y'|^{1/\beta}\}$ or, if $\alpha=\beta$, $d(x,y)=\|x-y\|^{1/\alpha}$ where $\|\cdot\|$ is any norm on $\R^2$.
It has been shown in \cite[Proposition 5.1]{MR3739202} that, for $\alpha=\beta=2$, there exists an $A$-homogeneous distance $d$ in $\R^2$ whose spheres are fractals.

If $\alpha \geq 1$, the maps $$\delta_\lambda: = \lambda^\alpha \begin{pmatrix} \cos(\log\lambda) & -\sin(\log\lambda)\\ \sin(\log\lambda) & \cos(\log\lambda)\end{pmatrix}
= \exp\left(\log(\lambda) \begin{pmatrix} \alpha & -1\\ 1 & \alpha\end{pmatrix}\right)
$$ are a one-parameter group of dilating automorphisms for the distance $d(x,y)=\|x-y\|^{\frac1\alpha}$, where $\|\cdot\|$ is the Euclidean norm.

If $\alpha=1$, one can show that the only homogeneous distances are multiples of the Euclidean distance.
This is a particular instance of a more general fact, see Proposition~\ref{prop12182017}.

However, if $\alpha=2$, there are examples of pathological distances, see next example. 


\subsection{Dilations with non-real spectrum}\label{sec12271414}	
Let 
\[
A=\begin{pmatrix}2&-1\\1&2\end{pmatrix} .
\]
We claim that the set $B:=\{(x,y)\in\R^2:\|(x,y)\|_\infty\le1\}$ is the unit ball of an $A$-homogeneous distance $d$ on $\R^2$, where $\|(x,y)\|_\infty=\max\{|x|,|y|\}$. 	
By Lemma~\ref{lem10011902}, we need to show that $B$ is $A$-convex for the claim to be true.
Let $(x,y),(\bar x,\bar y)\in B$ and $t\in(0,1)$.
Then
\begin{multline*}
\|t^A(x,y) + (1-t)^A(\bar x,\bar y)\|_{\infty}\\
\!\!=\!\!
\left\|\begin{pmatrix}
t^2(\cos(\log t)x-\sin(\log t)y) + (1-t)^2 (\cos(\log(1-t))\bar x - \sin(\log(1-t))\bar y ) \\
t^2(\sin(\log t)x+\cos(\log t)y) + (1-t)^2 (\sin(\log(1-t))\bar x + \cos(\log(1-t))\bar y )
\end{pmatrix}\right\|_{\infty}\\
\le t^2(|\cos(\log(t))|+|\sin(\log(t))|) + (1-t)^2 (|\cos(\log(1-t))| + |\sin(\log(1-t))| ) .
\end{multline*}
Set $f(t)$ to be the last expression: we need to show that $f(t)\le1$ for all $t\in(0,1)$.
Since $f(t)=f(1-t)$, we only need to show that $f(t)\le1$ for $t\in[1/2,1)$.
Notice that 
\[
f(t)\le h(t) := t^2(|\cos(\log(t))|+|\sin(\log(t))|) + 2(1-t)^2 .
\]
Moreover, for $t\in[1/2,1)$, we have $\log(t)\in[-\log(2),0]\subset[-\pi/4,0]$ and thus $|\cos(\log(t))|+|\sin(\log(t))| = \cos(\log(t)) - \sin(\log(t))$.
Now, $h(t)\le1$ for $t\in[1/2,1)$ because $h(1/2)\le1$, $h(1)\le1$ and $h$ is convex.
Indeed, one can compute on the interval $[1/2,1)$
\begin{align*}
h'(t) &= (\cos((\log(t))-3\sin((\log(t)))t + 4(t-1) , \\
h''(t) &= -2\cos((\log(t)) - 4\sin((\log(t)) + 4 ,
\end{align*}
where $h''(t)\ge -2+4 >0$.
The proof is complete.

\subsection{A distance with non-diagonalizable dilations}\label{ex:non-diagonal}
It is known, see \cite[Section~6]{MR2116315} and \cite{MR3180486}, that for all
 $\alpha>1$
  the maps 
  \[
  \delta_\lambda = \begin{pmatrix} \lambda^\alpha & \lambda^\alpha\log(\lambda) \\ 0 &\lambda^\alpha \end{pmatrix}= \exp\left(\log(\lambda) \begin{pmatrix} \alpha & 1\\  0 & \alpha\end{pmatrix}\right)
  \]
  form a one-parameter group of dilating automorphisms for some admissible distance $d_\alpha$ on $\R^2$
that is invariant under translations.
Such distances have the property that their conformal dimension is not realized.  
Consequently, these distances cannot be biLipschitz equivalent to homogeneous distance with diagonalizable dilating automorphisms.

\subsection{Automorphisms without distances}    \label{ex:without_distances}

We shall now show that for no  $\lambda>0$ 
there is  an admissible translation-invariant distance $d$ on $\R^2$ such that
 \[
 d(\delta x,\delta y)=\lambda d(x,y), \qquad\forall x,y\in \R^2,
 \]
where 
  \[
  \delta:=\delta_\lambda := \begin{pmatrix} \lambda & \lambda\log(\lambda) \\ 0 &\lambda \end{pmatrix}= \exp\left(\log(\lambda) \begin{pmatrix} 1 & 1\\  0 & 1\end{pmatrix}\right)
  \]
Notice that this statement can be deduced (by snowflaking a candidate $d$) from the fact that the conformal dimension of the distances in Example~\ref{ex:non-diagonal} is not attained.
However, the argument below is elementary enough to be worth showing it.

Let $y\neq0$ be such that $d((0,0),(0,y))\le 1$. 
For all $n,m\in\N$, we have
\begin{align*}
	d\left( 0,n\lambda^{m} (m\log(\lambda) y,y) \right)
	&= d(0,n\delta_{\lambda}^m (0,y))
	\le \sum_{k=1}^n d(0,(\delta_\lambda)^m (0,y) ) \\
	&= n \lambda^m d(0,(0,y)) 
	\leq n \lambda^m .
\end{align*}
Without loss of generality, we can assume $\lambda<1$.
For each $m\in \N$  we   take $n_m:=\lfloor \lambda^{-m}\rfloor$ and
look at the points 
$p_m : = n_m\lambda^{m} (m\log(\lambda) y,y)$.
On the one hand, the sequence $(p_m)_m $ diverges to infinity. On the other hand, from the calculation above it stays in the unit ball with respect to the distance $d$.  
This contradicts the fact that closed balls with respect to $d$ are compact (see Lemma~\ref{lem10011902}).

%

\subsection{Dilations that are not continuous}
If $\phi:\R\to\R^2$ is a $\Q$-linear group isomorphism (which exists, using the Axiom of Choice, because $\R$ and $\R^2$ are vector spaces over the rationals with the same dimension), then $d(x,y) = \|\phi(x)-\phi(y)\|$ is a left-invariant distance on the set $\R$ such that, for each $q\in\Q$, the map $x\mapsto q\cdot x $ is a dilating automorphism of factor $q$, but this distance is not admissible because $(\R,d)$ is isometric (and thus homeomorphic) to the standard $\R^2$.

Notice also that this distance on $\R$ is homothetic and that all its dilations fixing $0$ are group automorphisms, but some of them are not continuous on $\R$.

\subsection{Dilations that are not group automorphisms}
Let $G$ be the Lie group given by $\R^3$ with the group operation
\[
\begin{pmatrix}a\\b\\c\end{pmatrix}
*
\begin{pmatrix}x\\y\\z\end{pmatrix}
=
\begin{pmatrix}a\\b\\c\end{pmatrix}
+
\begin{pmatrix}\cos(c) & -\sin(c) & 0 \\ \sin(c) & \cos(c) & 0 \\ 0 & 0 & 1 \end{pmatrix}
\begin{pmatrix}x\\y\\z\end{pmatrix} .
\]
The group $G$ is the universal covering space of the rototranslation group $\R^2\ltimes\bb S^1$.
It is evident that the Euclidean distance $d_E$ on $\R^3$ is a left-invariant admissible distance on $G$.
The maps $\delta_\lambda p:=\lambda p$, $\lambda>0$, form a one-parameter group of diffeomorphisms of $G$ and $\delta_\lambda$ is a dilation of factor $\lambda$ for $d_E$.
So, $d_E$ is an admissible left-invariant homothetic distance on $G$.
But the dilations $\delta_\lambda$ are not group automorphisms of $G$ and $G$ is not nilpotent.

\subsection{Self-similar Lie group that is not homothetic}
Let $\rho:[0,\infty) \to [0,\infty)$
be the function whose upper graph is the convex hull of the points
$(2^{2m}, 2^m)$ as $m\in \Z$.
Define the translation-invariant distance on $\R$ such that
$d(0,t)=\rho(t)$, as $t>0$.
Then the map $t\to 4t$ is a metric dilation of factor 2. However, this
distance is not isometric to the Euclidean distance and it does not
admit dilations of every factor.

\section{When   $A$-homogeneous distances exist}\label{sec12061717} 
This section is devoted to the proof of Theorem~\ref{thm12041046}.
We start with two lemmas that allow us to modify homogeneous distances.
We will then prove $\ref{thm12041046item1}\THEN\ref{thm12041046item2}$ in Proposition~\ref{prop10011823}, while in Proposition~\ref{prop10031604} we shall prove $\ref{thm12041046item2}\THEN\ref{thm12041046item1}$.
Finally, notice that in the conditions of Theorem~\ref{thm12041046}.\ref{thm12041046item2}, 
the presence of a positive grading implies that 
$G$ is   nilpotent.

\subsection{New homogeneous distances from old ones}

The following lemma allows us to consider only derivations with real spectrum.
Recall that by $\sigma(K)$ we denote the spectrum of an endomorphism $K$. We shall denote by $\g$ the Lie algebra of a Lie group $G$.
\begin{lemma}\label{lem10222333}
	Let $A\in\Der(\frk g)$ and $d$ be an $A$-homogeneous distance on $G$.
	Let $K\in\Der(\frk g)$ be such that $\sigma(K)\subset i\R$, $K$ is diagonalizable over $\C$, and $[A,K]=0$.
	Then there is a distance $d'$ that is $(A+K)$-homogeneous, $\lambda^K$-invariant, and biLipschitz equivalent to $d$.
\end{lemma}
\begin{proof}
	Since $\sigma(K)\subset i\R$ and $K$ is diagonalizable, then $\scr K:= \overline{\{\lambda^K\}_{\lambda>0}}$ is a compact subgroup of $\Aut(\frk g)$.
	Since $[A,K]=0$, then  $\lambda^A\mu^K=\mu^K\lambda^A$, for all $\lambda,\mu>0$.
	Therefore, by Lemma~\ref{lem10290953}, there is a distance $d'$ on $G$ that is both $A$-homogeneous and $\scr K$-invariant.
	
	From $[A,K]=0$, we also get $\lambda^{A+K}=\lambda^A\lambda^K$.
	Since $\lambda^{A+K}$ is the composition of a dilation with an isometry of $d'$, then $d'$ is also $(A+K)$-homogeneous.
	Since both $d$ and $d'$ share a nontrivial dilation, then the identity $(G,d)\to (G,d')$ is biLipschitz by Lemma~\ref{lem10291010}
\end{proof}
The following lemma is a variation of Lemma~\ref{lem12041918}. The proof is left to the reader.

\begin{lemma}\label{lem07251552}
	Let $A\in\Der(\frk g)$.
	If there is an $A$-homogeneous distance on $G$ and if $\frk h\lhd\frk g$ is an ideal with $A(\frk h)\subset\frk h$, then there is an $\hat A$-homogeneous distance on $G/H$, where $\hat A\in\Der(\frk g/\frk h)$ is induced by $A$ and $H=\exp(\frk h)$.
\end{lemma}
	
\subsection{Necessary condition for $A$-homogeneous distances}
Here we prove that~\ref{thm12041046item1} implies~\ref{thm12041046item2} in Theorem~\ref{thm12041046}.

Let $A$ be a derivation on the Lie algebra $\frk g$ of a Lie group $G$.
 Let $\frk g = \bigoplus_{t\in\R} V_t$ be the real grading defined by $A$ as in Proposition~\ref{prop12061739}.
Suppose that there is an $A$-homogeneous distance on $G$.
Then $G$ is connected simply connected, $V_t=\{0\}$ for all $t<1$, and $\frk g$ is nilpotent by Theorem~\ref{teo08221641}.


\begin{proposition}\label{prop10011823}
	Let $G$ be a  Lie group equipped with  an $A$-homogeneous distance, for some derivation $A$.
	Then $A|_{V_1}$ is diagonalizable over $\C$. 
\end{proposition}
\begin{proof}
	 By Corollary~\ref{cor07251444} and Lemma~\ref{lem10222333}, we can assume that the eigenvalues of $A$ are all real, because $(A-A_I)|_{V_1}$ is diagonalizable if and only if $A|_{V_1}$ is. 
	By Theorems~\ref{thm11171734} and~\ref{teo08221641}, we have $\frk g = V_1\oplus\bigoplus_{s>1} V_s$, with $V_1\neq\{0\}$.
	Arguing by contradiction, suppose that $A|_{V_1}$ is not diagonalizable.
	
	Let $b_1,\dots,b_r\in V_1$ be a basis so that the matrix representation of $A|_{V_1}$ with respect to this basis is in Jordan normal form.
	Since $A|_{V_1}$ is not diagonalizable, we can assume $A(b_r)=b_r+b_{r-1}$ and $A(\Span_\R\{b_1,\dots,b_{r-2}\})\subset\Span_\R\{b_1,\dots,b_{r-2}\}$.
	Let $\frk h=\Span_\R\{b_1,\dots,b_{r-2}\}\oplus\bigoplus_{s>1}V_s$.
	Then $\frk h$ is an ideal of $\frk g$ and $A(\frk h)\subset\frk h$.
	Therefore, by Lemma~\ref{lem07251552}, 
	there is a $\hat A$-homogeneous distance
	on the quotient group $\hat G:=G/\exp(\frk h)\simeq\Span_\R(b_{r-1},b_r)$ 
	where $\hat A=\begin{pmatrix}1&1\\0&1\end{pmatrix}$ in the basis $(b_{r-1},b_r)$.
	However, we showed in Example~\ref{ex:without_distances} that such a distance does not exist.
\end{proof}

\subsection{Construction of an $A$-homogeneous distance}
Here we prove that~\ref{thm12041046item2} implies~\ref{thm12041046item1} in Theorem~\ref{thm12041046}.

Let $G$ be a connected simply connected nilpotent  Lie group with Lie algebra  $\frk g$ and 
let $A$ be a derivation on $\frk g$.
Let $\frk g = \bigoplus_{t\ge1} V_t$ be the real grading defined by $A$ as in Proposition~\ref{prop12061739}.
Since $G$ is simply connected and nilpotent, the exponential map $\frk g\to G$ is a diffeomorphism.
For simplicity in the exposition, we will identify $\frk g$ and $G$ via the exponential map.
Via this identification, the Lie algebra automorphism $\lambda^A$ of $\frk g$ is a Lie group automorphism of $G$, for all $\lambda>0$.

\begin{lemma}\label{lem08122132}
	Let $A\in\Der(\frk g)$.
	For $t>0$, define
	\[
	W_t := \bigoplus_{s\in\R} E^A_{t+is} \subset\frk g_\C  ,
	\]
	so that $V_t=\g\cap W_t$. 
	For every $\theta\in(0,1)$ there is a norm $\|\cdot\|$ on $\frk g_\C$ such that the following holds: For all $t>0$, if $W\subset W_t$ is such that $AW\subset W$, then for all $\lambda\in[0,1]$
	\begin{align}
	\label{eq08122104a} 
		\|\lambda^A|_W\| &\le \lambda^{t-\theta}\qquad \text{or} \\
	\label{eq08122104b} 
		\|\lambda^A|_W\| &\le \lambda^t \qquad\text{ if $A|_W$ is diagonalizable over $\C$} ,
	\end{align}
	where $\|\lambda^A|_W\|$ is the operator norm of the linear operator $\lambda^A|_W:(W,\|\cdot\|)\to (W,\|\cdot\|)$.

	Moreover, the norm $\|\cdot\|$ can be defined by an Hermitian product on $\frk g_\C$ for which the spaces $W_t$ are orthogonal to each other.
\end{lemma}
\begin{proof}
	Let $(b_1,\dots,b_n)$ be a basis of $\frk g_\C$ such that the matrix representation of $A$ is in Jordan normal form.
	In other words, the matrix $M$ of $A$ in this basis has the eigenvalues of $A$ on the diagonal, some 1 on the upper diagonal and 0 in all the other entries.
	For every $\epsilon>0$, define a new basis $b^\epsilon_1,\dots,b^\epsilon_n$ with $b^\epsilon_j:=\epsilon^jb_j$.
	Then, the matrix $M^\epsilon$ of $A$ in this new basis is the same as $M$, but the 1 in the upper diagonal are replaced with $\epsilon$.
	Indeed, on the one hand, if $Ab_j = M_{jj}b_j$, then $Ab^\epsilon_j=M_{jj} b^\epsilon_j$; 
	On the other hand, if $A b_j = M_{jj} b_j + b_{j-1}$, then 
	\[
	Ab^\epsilon_j = \epsilon^j (M_{jj} b_j + b_{j-1}) 
	= M_{jj} b^\epsilon_j + \epsilon b^\epsilon_{j-1} .
	\]
	Notice that the nilpotent part of $A$, i.e., the linear map $A_N\in\Der(\frk g_\C)$ defined in Corollary~\ref{cor07251444}, is represented by the matrix $M^\epsilon_N$ that is $M^\epsilon$ with the diagonal entries replaced by $0$. 
	
	Let $\langle , \rangle_\epsilon$ be the Hermitian form on $\frk g_\C$ such that $b^\epsilon_1,\dots,b^\epsilon_n$ are orthonormal and let $\|\cdot\|_\epsilon$ be the corresponding norm.
	Then the operator norm $\|A_N\|_\epsilon = \|M^\epsilon_N\|$ is arbitrarily small as $\epsilon\to0^+$.
	
	For reasons that will appear evident shortly, we need the following fact: 
	there is $\epsilon>0$ such that 
	\[
	f_\epsilon(\lambda):=\lambda^{\theta} +  \sum_{j=1}^m (-1)^j \frac{\lambda^{\theta}\log(\lambda)^j}{j!} \|A_N\|_\epsilon^j \le 1,
	\text{ for all $\lambda\in[0,1]$,}
	\]
	where $m\in\N$ is such that $A_N^m=0$.
	Indeed, first of all notice that $f_\epsilon(0)=0$ and $f_\epsilon(1)=1$.
	Next, if $\lambda\in[0,1/2]$, then $f_\epsilon(\lambda)\le 1$ if $\|A_N\|_\epsilon$ is small enough.
	Finally, if $\lambda\in[1/2,1]$, then $f_\epsilon$ is smooth with first derivative as close as wished to $\theta \lambda^{\theta-1}$, as $\epsilon\to0^+$.
	Since $\theta \lambda^{\theta-1}\ge\theta 2^{-|\theta-1|}>0$, then, if $\epsilon>0$ is small enough, $f_\epsilon'(\lambda)>0$ for all $\lambda\in[1/2,1]$.
	Thus $f_\epsilon(\lambda)\le1$ for $\lambda\in[1/2,1]$, because $f_\epsilon(1)=1$.
	
	The claim is proven: we fix such an $\epsilon>0$.	
	Fix a subspace $W\subset W_t$ such that $AW=W$.
	Let $A_R,A_I,A_N\in\Der(\frk g_\C)$ be as in Corollary~\ref{cor07251444}.
	Then $\|\lambda^{A_I}|_W\|_\epsilon=1$, because the matrix representation of $A_I|_W$ is diagonal with purely imaginary entries.
	Since $A_R|_W=t\Id$,  then $\|\lambda^{A_R}|_{W}\|_\epsilon=\lambda^t$.
	For all $\lambda\in[0,1]$, we have
	\begin{align*}
	\|\lambda^{A}|_{W}\|_\epsilon
	&\le \|\lambda^{A_R}|_{W}\|_\epsilon \cdot \|\lambda^{A_I}|_{W}\|_\epsilon \cdot \|\lambda^{A_N}|_{W}\|_\epsilon 
	= \lambda^t \, \left\| \Id|_W + \sum_{j=1}^m \frac{\log(\lambda)^j}{j!} (A_N|_W)^j\right\|_\epsilon \\
	&\le \lambda^{t-\theta} \, \left(\lambda^{\theta} +  \sum_{j=1}^m \frac{\lambda^{\theta}|\log(\lambda)|^j}{j!} \|A_N\|_\epsilon^j\right) 
	= \lambda^{t-\theta} f_\epsilon(\lambda) 
	\le \lambda^{t-\theta} ,
	\end{align*}
	where the first inequality uses the sub-multiplicity of operator norms.
	The estimate~\eqref{eq08122104a} is thus proven.
	
	If $A|_W$ is diagonalizable over $\C$, i.e., $A_N|_W=0$, then 
	$\|\lambda^A|_{W}\|_\epsilon = \lambda^t$ and thus estimate~\eqref{eq08122104b} is also proven.
\end{proof}

\begin{lemma}\label{lem10011955}
	In the hypothesis of Theorem~\ref{thm12041046}.\ref{thm12041046item2}, assume that 
	$\frk g$ is Abelian.
	Then there is an $A$-homogeneous distance.
\end{lemma}
\begin{proof}
	Notice that, after the identification $G=\frk g$, the group operation of $G$ is just the vector sum in $\frk g$.
	Let $\|\cdot\|$ be a norm given by Lemma~\ref{lem08122132} with $\theta>0$ such that $t-\theta>1$ for all $t>1$ with $V_t^A\neq\{0\}$, or, equivalently, $W_t\neq\{0\}$.
	Define 
	\(
	B=\{v\in\frk g:\|v\|\le1\} 
	\).
	Since $B$ trivially satisfies both conditions \ref{lem10011902item1} and \ref{lem10011902item2} of Lemma~\ref{lem10011902}, we only need to show that $B$ is $A$-convex.
	
	First, we claim that for all $\lambda\in(0,1)$ and $x\in\frk g$,
	\begin{equation}\label{lem10011907}
	\|\lambda^Ax\| \le \lambda\|x\| .
	\end{equation}
	Indeed, because the decomposition $\frk g=\bigoplus_{t\ge1}V_t$ is orthogonal with respect to the scalar product that defines $\|\cdot\|$, we obtain from Lemma~\ref{lem08122132}
	\[
	\|\lambda^Ax\|^2 
	= \sum_{t\ge1}\|\lambda^Ax_t\|^2
	\le \lambda^2\|x_1\|^2 + \sum_{t>1} \lambda^{2(t-\theta)} \|x_t\|^2 
	\le \lambda^2 \sum_{t\ge1} \|x_t\|^2
	= \lambda^2 \|x\|^2 ,
	\]
	where $x=\sum_{t\ge1} x_t\in\frk g$, $x_t\in V_t$ and $\lambda\in(0,1)$.
	So, we have obtained~\eqref{lem10011907}.
	
	Next, if $x,y\in B$ and $\lambda\in(0,1)$, then we get from~\eqref{lem10011907}
	\begin{align*}
	\|(\lambda^Ax)((1-\lambda)^Ay)\| 
	&= \|\lambda^Ax+(1-\lambda)^Ay\| \\
	&\le \|\lambda^Ax\| + \|(1-\lambda)^Ay\| 
	\le \lambda \|x\| + (1-\lambda) \|y\| 
	\le 1 .
	\end{align*}
	Therefore, $B$ is $A$-convex and so it is the unit ball of an $A$-homogeneous distance on $G$ by Lemma~\ref{lem10011902}.
\end{proof}

\begin{lemma}\label{lem10011942}
	Let $\chi_C:[0,1]\to\R$ be the function
	\begin{align*}
	\chi_C(t) &= t^2\max\{|\log(t)|,|\log(t)|^n\} \\
		& \qquad+ (1-t)^2\max\{|\log(1-t)|,|\log(1-t)|^n\} \\
		& \qquad\qquad- C t(1-t),
	\end{align*}
	where $n\in\N$.
	Then there is $C>0$ such that $\chi_C(t)\le 0$ for all $t\in[0,1]$.
\end{lemma}
\begin{proof}
	Notice that, if $t\in(0,1)$, $|\log(t)|^n 
	= \max\{|\log(t)|,|\log(t)|^n\}$ if and only if $|\log(t)|\ge 1$, i.e., if and only if $t\in(0,e^{-1}]$;
	Similarly, $|\log(1-t)|^n 
	= \max\{|\log(1-t)|,|\log(1-t)|^n\}$ if and only if $|\log(1-t)|\ge1$, i.e., if and only if $t\in[1-e^{-1},1)$.
	Therefore
	\[
	\chi_C(t) = 
	\begin{cases}
	\chi_C^1(t) &\text{if } t\in(0,e^{-1}] , \\
	\chi_C^2(t) &\text{if } t\in[e^{-1},1-e^{-1}], \\
	\chi_C^3(t) &\text{if } t\in[1-e^{-1},1) , \\
	\end{cases}
	\]
	where
	\begin{align*}
	\chi_C^1(t) &= t^2 \log\left(\frac1t\right)^n + (1-t)^2 \log\left(\frac1{1-t}\right) - Ct(1-t) , \\
	\chi_C^2(t) &= t^2 \log\left(\frac1t\right) + (1-t)^2 \log\left(\frac1{1-t}\right) - Ct(1-t) , \\
	\chi_C^3(t) &= t^2 \log\left(\frac1t\right) + (1-t)^2 \log\left(\frac1{1-t}\right)^n - Ct(1-t) .
	\end{align*}
	
	We prove the lemma in each of the three intervals.
	
	Case 1: $\chi_C^1(t)\le 0$ for $t\in(0,e^{-1}]$ and $C$ large enough.
	Notice that $\lim_{t\to0^+}\chi^1_C(t)=0$ and that 
	\begin{align*}
	(\chi_C^1)'(t) &= f(t) + C(2t-1)
		,\text{ where }\\
	f(t) &=
		t\log\left(\frac1t\right)^{n-1} \left(2\log\left(\frac1t\right) - n\right) 
		+(1-t) \left( - 2\log\left(\frac1{1-t}\right) + 1 \right) .
	\end{align*}
	Since $e^{-1}<1/2$ ($e=2.719\dots$), then $(2t-1)<-\epsilon$ for some $\epsilon>0$.
	Since, $f$ is a smooth function on $(0,e^{-1}]$ with $\lim_{t\to0^+}f(t) = 1$, then $f$ is bounded on $(0,e^{-1}]$, say $\sup_{(0,e^{-1}]}f(t)\le M$.
	Therefore, there is $C>0$ large so that 
	\(
	(\chi_C^1)'(t)\le M - \epsilon C \le 0
	\)
	for all $t\in(0,e^{-1}]$.
	Hence, $\chi_C^1$ is a decreasing function with $\chi_C^1(0)=0$, and thus $\chi_C^1(t)\le 0$ for all $t\in(0,e^{-1}]$.
	
	Case 2: $\chi_C^2(t)\le 0$ for $t\in[e^{-1},1-e^{-1}]$ and $C$ large enough.
	In this case we have 
	\[
	\chi_C^2(e^{-1}) = \chi_C^2(1-e^{-1}) = e^{-2} - (1-e^{-1})^2 \log(1-e^{-1}) - C e^{-1} (1-e^{-1}) 
	\]
	and 
	\[
	(\chi_C^2)''(t) 
	= 2C-2(\log(1-t)+\log(t) + 3) .
	\]
	Since $e^{-1} (1-e^{-1})>0$ and since $-2(\log(1-t)+\log(t) + 3)$ is a smooth function on $[e^{-1},1-e^{-1}]$,
	then there is $C>0$ such that
	$\chi_C^2(e^{-1}) = \chi_C^2(1-e^{-1}) <0$ and $(\chi_C^2)''\ge0$ on $[e^{-1},1-e^{-1}]$.
	We conclude that $\chi_C^2(t)\le 0$ for all $t\in[e^{-1},1-e^{-1}]$.
	
	Case 3: $\chi_C^3(t)\le 0$ for $t\in[e^{-1},1-e^{-1}]$ and $C$ large enough.
	Since $\chi_C^3(t) = \chi_C^1(1-t)$, this case follows from Case 1.
\end{proof}

\begin{lemma}\label{lem10011958}
	In the hypothesis of Theorem~\ref{thm12041046}.\ref{thm12041046item2}, assume that 
	$V_t=\{0\}$ for $t>2$.
	Then there is an $A$-homogeneous distance.
\end{lemma}
\begin{proof}
	By Corollary~\ref{cor07251444} and Lemma~\ref{lem10222333}, we can assume that the spectrum of $A$ is real.
	In particular, $A|_{V_1} = \Id|_{V_1}$.
	If $V_2=\{0\}$, then the thesis follows from Lemma~\ref{lem10011955}.
	So, we assume that $V_2$ is nontrivial.
	
	Let $\{b_{j,k}:k=1,\dots,m,\ j=0,\dots,n_k\}$ be a basis of $V_2$ such that the matrix representation of $A$ in this basis is in Jordan normal form and such that, for each $k\in\{1,\dots,m\}$ the vectors $b_{0,k},\dots,b_{n_k,k}$ form a basis for one Jordan block.
	Define 
	\[
	W := \Span_{\R} \{b_{0,k}:k=1,\dots,m\} \subset V_2 .
	\]
	The vector space $W$ is the largest subspace of $V_2$ on which $A$ is $\R$-diagonalizable and $Aw=2w$ for all $w\in W$.
	Moreover, since $A$ is $\R$-diagonalizable on $V_1$ and Lie brackets of eigenvectors are eigenvectors, then 
	\begin{equation}
	[\frk g,\frk g] = [V_1,V_1] \subset W .
	\end{equation}
	Let $\langle \cdot,\cdot \rangle$ be a scalar product on $\frk g$ such that the spaces $V_t$ are orthogonal to each other and such that $\{b_{j,k}\}_{j,k}$ is an orthonormal basis of $V_2$.
	We denote by $\pi_W$ the orthogonal projection $\frk g\to W$.
	If $x\in\frk g$, we denote by $x_1$, $x_2$ and $x_W$ the orthogonal projections of $x$ in $V_1$, $V_2$ and $W$, respectively.
	
	We claim that there is $C>0$ such that the following holds: 
	If $x,y\in\frk g$ are such that $\|x-x_W\|\le 1$, $\|y-y_W\|\le 1$, $\|x_W\|\le C$ and $\|y_W\|\le C$, then, for all $\lambda\in(0,1)$, 
	\begin{equation}\label{eq10011942}
	\|\pi_W(\lambda^Ax (1-\lambda)^Ay)\| \le C .
	\end{equation}
	
	First, if $x_2 = \sum_{k=1}^m\sum_{j=0}^{n_k} x_2^{j,k} b_{j,k}$ and $\lambda>0$, then
	\begin{align*}
	\pi_W(\lambda^Ax_2)
	&= \lambda^2 \sum_{k=1}^m \left(\sum_{j=0}^{n_k} \frac{\log(\lambda)^j}{j!} x_2^{j,k} \right) b_{0,k} \\
	&= \lambda^2 \pi_W(x_2) + \lambda^2 \sum_{k=1}^m \left(\sum_{j=1}^{n_k} \frac{\log(\lambda)^j}{j!} x_2^{j,k} \right) b_{0,k}  .
	\end{align*}
	Therefore, if $\|x_2-\pi_W(x_2)\|\le 1$, i.e., $|x_2^{j,k}|\leq 1$ for $j\neq0$, then
	\begin{equation}\label{eq10011941}
	\|\pi_W(\lambda^Ax_2)\| \le \lambda^2 \|\pi_W(x_2)\| + \lambda^2 n \max\{|\log(\lambda)|,|\log(\lambda)|^n\} ,
	\end{equation}
	where $n=\dim V_2$.
	
	Second, if $x,y\in\frk g$ and $\lambda\in(0,1)$, then 
	\begin{align*}
	(\lambda^Ax)((1-\lambda)^Ay)
	&= \lambda^Ax + (1-\lambda)^Ay + \frac12 [\lambda^Ax_1,(1-\lambda)^Ay_1] \\
	&= \lambda^Ax + (1-\lambda)^Ay + \frac{\lambda(1-\lambda)}{2} [x_1,y_1] ,
	\end{align*}
	because of the Baker–Campbell–Hausdorff formula, the fact that $[\frk g,\frk g]=[V_1,V_1]$, being $V_t=\{0\}$ for $t>2$, and the hypothesis that $A$ is diagonal on $V_1$.
	
	
	Third, let $C$ be such that $\|[x_1,y_1]\| \le C \|x_1\|\,\|y_1\|$ for all $x_1,y_1\in V_1$, which exists because $[\cdot,\cdot]$ is a bilinear map.
	Suppose $x,y\in\frk g$ and $\lambda\in(0,1)$ are such that $\|x-x_W\|\le 1$, $\|y-y_W\|\le 1$, $\|x_W\|\le C$ and $\|y_W\|\le C$.
	Then $\|x_1\|\le1$, $\|y_1\|\le1$ and
	\begin{align*}
	\|\pi_W(\lambda^Ax (1-\lambda)^Ay)\| 
	&= \left\| \pi_W\left( \lambda^A x_2 \right) + \pi_W\left( (1-\lambda)^A y_2 \right) + \frac{\lambda(1-\lambda)}{2} [x_1,y_1] \right\| \\
	&\le \|\pi_W\left( \lambda^A x_2 \right)\| + \|\pi_W\left( (1-\lambda)^A y_2 \right)\| + \frac{\lambda(1-\lambda)}{2} C \|x_1\|\|y_1\| \\
	&\le \lambda^2 \|x_W\| + \lambda^2 n \max\{|\log(\lambda)|,|\log(\lambda)|^n\} \\
	&\qquad	+ (1-\lambda)^2 \|y_W\| \\
	& \qquad\qquad+ (1-\lambda)^2 n \max\{|\log((1-\lambda))|,|\log((1-\lambda))|^n\} \\
	&\qquad\qquad\qquad	+ \frac{\lambda(1-\lambda)}{2} C \\
	&\le C + n \bigg(
		\lambda^2  \max\{|\log(\lambda)|,|\log(\lambda)|^n\} \\
	&\qquad	+ (1-\lambda)^2  \max\{|\log((1-\lambda))|,|\log((1-\lambda))|^n\} \\
	&\qquad\qquad	- \lambda(1-\lambda)\frac{3C}{2n}
		\bigg) ,
	\end{align*}
	where we used~\eqref{eq10011941} in the second last inequality.
	Finally, by Lemma~\ref{lem10011942},  if $C$ is large enough, then the second term of the upper bound is non-positive and thus we obtain the claim~\eqref{eq10011942}.
	
	We are now in the position to conclude the proof.
	Since $[\frk g,\frk g]\subset W$, 
	then $\hat{\frk g}=\frk g/W$ is an Abelian Lie algebra.
	Since $\hat{\frk g}$ is nilpotent and the corresponding group quotient $\hat G:=G/\exp(W)$ is simply connected, we will identify $\hat G$ with $\hat{\frk g}$.
	Denote by $\pi:\frk g\to\hat{\frk g}$ the quotient map.
	Since $A(W)\subset W$, the derivation $A$ induces $\hat{A}\in\Der(\hat{\frk g})$ with $\hat{A}\pi=\pi A$.
	
	By Lemma~\ref{lem10011955}, there is $\hat{B}\subset\hat{\frk g}$ that is the unit ball of a $\hat{A}$-homogeneous distance.
	Let $W^\perp$ be the orthogonal complement of $W$ in $\frk g$.
	Define $\hat{B}'= \pi^{-1}(\hat{B})\cap W^\perp$.
	Since any $\hat{A}$-homogeneous distance induces the manifold topology by Theorem~\ref{thm11171734} and since $\pi:W^\perp\to\hat{\frk g}$ is a linear isomorphism, we may assume that 
	\[
	\hat{B}'\subset \{x\in W^\perp:\|x\|\le 1\} \subset\{x\in\frk g:\|x-x_W\|\le 1\} .
	\]
	Define
	\begin{align*}
	B &= \{x\in\frk g : \pi(x)\in \hat{B}, \ \|x_W\|\le C\} \\
	&= \{x\in\frk g:x-x_W\in \hat{B}',\ \|x_W\|\le C\} .
	\end{align*}
	We shall prove that $B$ is the unit ball of an $A$-homogeneous distance.
	We do this by means of Lemma~\ref{lem10011902}: The only non-trivial property we need to check is $A$-convexity.
	Let $x,y\in B$ and $\lambda\in(0,1)$.
	On the one hand, 
	\[
	\pi(\lambda^Ax(1-\lambda)^Ay) = \lambda^{\hat{A}}\pi(x) (1-\lambda)^{\hat{A}}\pi(y) \in \hat{B} ,
	\]
	because $\hat{B}$ is $\hat{A}$-convex.
	On the other hand, by~\eqref{eq10011942},
	\[
	\|\pi_W(\lambda^Ax(1-\lambda)^Ay)\| \le C .
	\]
	So we constructed an $A$-homogeneous distance on $G$.
\end{proof}

\begin{proposition}\label{prop10031604}
	In the hypothesis of Theorem~\ref{thm12041046}.\ref{thm12041046item2},
	there is an $A$-homogeneous distance on $G$.
\end{proposition}
\begin{proof}
	We shall prove the proposition by induction on the number of non-trivial layers $N=\#\{t\ge1:V_t\neq\{0\}\}$.
	If $N=1$, then $\frk g$ is Abelian, so we have the thesis from Lemma~\ref{lem10011955}.
	
	Assume that the thesis holds for graded Lie algebras with $N$ layers and let $\frk g=\bigoplus_{j=1}^{N+1} V_{t_j}$ have $N+1$ layers, where $1\le t_1<t_2<\dots<t_{N+1}$.
	If $t_{N+1}\le 2$, then the thesis holds by Lemma~\ref{lem10011958}.
	
	Suppose that $t_{N+1}>2$.
	If $x\in\frk g$, we denote by $x_j$ the component in $V_{t_j}$ of $x$, and $\bar x=x-x_{N+1}=\sum_{j=1}^N x_j$.
	Let $\theta\in(0,1)$ be such that $t_{N+1}-\theta>2$ and $t_j-\theta>1$ for all $t_j>1$.
	Let $\|\cdot\|$ be a norm on $\frk g$ given by Lemma~\ref{lem08122132} with this $\theta$.
	
	Since $V_{t_{N+1}}$ is an ideal of $\frk g$, then $\hat{\frk g}=\frk g/V_{t_{N+1}}$ is a Lie algebra.
	Since $\hat{\frk g}$ is nilpotent and the corresponding group quotient $\hat{G}:=G/\exp(V_{t_{N+1}})$ is simply connected, we will identify $\hat{G}$ with $\hat{\frk g}$.
	Denote by $\pi:\frk g\to\hat{\frk g}$ the quotient map.
	Since $A(V_{t_{N+1}})\subset V_{t_{N+1}}$, the derivation $A$ induces $\hat{A}\in\Der(\hat{\frk g})$ with $\hat{A}\pi=\pi A$.
	
	By the inductive hypothesis, there is $\hat{B}\subset\hat{\frk g}$ that is the unit ball of a $\hat{A}$-homogeneous distance.
	Let $\hat{B}'=\pi^{-1}(\hat{B})\cap \bigoplus_{j=1}^N V_{t_j}$.
	Since the restriction $\pi:\bigoplus_{j=1}^N V_{t_j}\to\hat{\frk g}$ is a linear isomorphism, we can assume that
	\begin{equation}\label{eq10012014}
	\hat{B}'\subset\left\{x\in\frk g: \sum_{j=1}^N\|x_j\| \le 1\right\} .
	\end{equation}
	
	If $\bar x\in\hat{B}'$ and $\lambda\in(0,1)$, then 
	\begin{equation}\label{eq12052225}
	\|\lambda^A\bar x\| \le \sum_{j=1}^N \|\lambda^Ax_j\| \le \lambda\sum_{j=1}^N\|x_j\| \le \lambda ,
	\end{equation}
	because of~\eqref{eq08122104b} and the fact that $A$ is $\R$-diagonal on $V_1$, because of~\eqref{eq08122104a} together with $t_j-\theta>1$ for $t_j>1$, and also by~\eqref{eq10012014}.
	
	Notice that if $x,y\in\frk g$, then 
	\[
	(xy)_{N+1} = x_{N+1} + y_{N+1} + P_{N+1}(\bar x,\bar y),
	\]
	where $P_{N+1}$ has polynomial components in any system of linear coordinates.
	Since $P_{N+1}(0,\bar y) = P_{N+1}(\bar x,0)=0$ and $\hat{B}'$ is compact, there is $C>0$ such that
	\begin{equation}\label{eq12052205}
	\|P_{N+1}(\bar x,\bar y)\| \le 2C \|\bar x\|\,\|\bar y\|
	\qquad\forall \bar x\bar y\in \hat{B}' .
	\end{equation}
	
	We claim that, if $C>0$ is given by~\eqref{eq12052205}, then 
	\begin{align*}
	B &:= \{x\in\frk g:\pi(x)\in \hat{B},\ \|x_{N+1}\|\le C\} \\
	&= \{x\in\frk g:\bar{x}\in \hat{B}',\ \|x_{N+1}\|\le C\} 
	\end{align*}
	is the unit ball of an $A$-homogeneous distance.	
	We prove our claim by means of Lemma~\ref{lem10011902}: The only non-trivial condition we need to prove is $A$-convexity of $B$.
	
	Let $x,y\in B$ and $\lambda\in(0,1)$.
	On the one hand, 
	\[
	\pi(\lambda^Ax(1-\lambda)^Ay) = \lambda^{\hat{A}}\pi(x) (1-\lambda)^{\hat{A}}\pi(y) \in \hat{B} ,
	\]
	because $\hat{B}$ is $\hat{A}$-convex.
	On the other hand
	\begin{align*}
	\|(\lambda^Ax(1-\lambda)^Ay)_{N+1}\|
	&= \left\| \lambda^Ax_{N+1} + (1-\lambda)^Ay_{N+1} + P_{N+1}(\lambda^A\bar x,(1-\lambda)^A\bar y) \right\| \\
	&\le \lambda^2 \|x_{N+1}\| + (1-\lambda)^2\|y_{N+1}\| + 2C \|\lambda^A\bar x\|\,\|(1-\lambda)^A\bar y\| \\
	&\le C(\lambda^2+(1-\lambda)^2+2\lambda(1-\lambda) )
	= C ,
	\end{align*}
	where we used in the first inequality the facts~\eqref{eq08122104a} and $t_{N+1}-\theta>2$, and~\eqref{eq12052225} in the second inequality.
	This completes the proof.
\end{proof}

\section{BiLipschitz reduction to real $A$-homogeneous distances}\label{sec12061813} 
This section is devoted to Theorems~\ref{thm12061648} and~\ref{thm12061651}.
Before diving into the proofs, we prove two preliminary lemmas in Section~\ref{sec12061857}.
The proofs of the theorems will be in the subsequent subsections.

\subsection{Algebraic preliminaries on the image of the exponential map}\label{sec12061857}
\begin{lemma}\label{lem10271051}
	Fix $\bb K\in\{\R,\C\}$.
	Let $\frk g$ be a Lie algebra over $\bb K$ and $A:\frk g\to\frk g$ a $\bb K$-linear map such that $e^A\in\Aut_{\bb K}(\frk g)$.
	If $A$ is nilpotent, then $A\in\Der_{\bb K}(\frk g)$.
\end{lemma}
\begin{proof}
	Let $N\in\N$ be such that $A^{N+1}=0$.
	For every $m\in\Z$, we have $e^{mA}\in\Aut_{\bb K}(\frk g)$.
	Therefore,
	expanding the exponential's series in the identity $e^{mA}[x,y]=[e^{mA}x,e^{mA}y]$,
	one can show that, for every $x,y\in\frk g$ and all $m\in\Z$
	\begin{equation}\label{eq10271035}
	\sum_{n=0}^N \frac{m^n}{n!} A^n[x,y]
	= \sum_{n=0}^{2N} \frac{m^n}{n!} \sum_{k=0}^n \binom{n}{k} [A^kx,A^{n-k}y] .
	\end{equation}
	Since these are polynomials in $m$ that coincide on $\Z$, then they have the same coefficients.
	In particular, the terms of order $n=1$ are
	\[
	A[x,y] = [Ax,y] + [x,Ay] . \qedhere
	\]
\end{proof}

\begin{lemma}\label{lem10271236}
	Let $\frk g$ be a real Lie algebra, $\phi\in\Aut(\frk g)$ and $\lambda\in(0,\infty)\setminus\{1\}$.
	Then there are $A\in\Der(\frk g)$ and $K\in\Aut(\frk g)$ such that, 
	\begin{enumerate}
	\item
	$\phi = K\lambda^A$;
	\item
	$K$ is $\C$-diagonalizable and $\sigma(K)\subset\bb S^1$;
	\item
	$\sigma(A)\subset\R$;
	\item
	$[K,A]=0$.
	\end{enumerate}
\end{lemma}
\begin{proof}
	Without loss of generality, we assume $\lambda=e$.
	Define $k,r,n:\C^*\to\C^*$ as
	\[
	k(\alpha) = \frac{\alpha}{|\alpha|}, \qquad
	r(\alpha) = |\alpha|, \qquad
	n(\alpha) = \frac1\alpha .
	\] 
	Consequently, with the terminology introduced just before Lemma~\ref{lem10271029}, define the linear maps $K=\phi_k$, $R=\phi_r$ and $N=\phi_n\circ\phi$ on $\frk g_\C$.
	By Lemma~\ref{lem10271029}, since the function $k$, $r$ and $n$ are multiplicative and they commute with the complex conjugation, then $K,R,N\in\Aut_\C(\frk g_\C)\cap\Aut(\frk g)$ and they commute with each other and with $\phi$.
	Moreover, $K$ is diagonalizable and $\sigma(K)\subset\bb S^1$.
		
	Since $r$ is a positive function, then we consider 
	\[
	\tilde A:= \phi_{\log\circ r}, \text{ so that }  R=e^{\tilde A} .
	\]
	We claim that $\tilde A \in\Der(\frk g)$.
	First, since $\log(r(\bar\alpha))=\overline{\log(r(\alpha))}$, for all $\alpha\in\C$, then $\tilde A (\frk g)\subset\frk g$.
	Second, if $v=\sum_\alpha v_\alpha$ and $w=\sum_\alpha w_\alpha$, where $v_\alpha,w_\alpha\in E^\phi_\alpha$, then $[v_\alpha,w_\beta]\in E^\phi_{\alpha\beta}$ by Lemma~\ref{lemBurba} and thus
	\begin{align*}
	\tilde A [v,w]
	&= \sum_{\alpha,\beta\in\sigma(\phi)} \log(|\alpha\beta|)[v_\alpha,w_\beta] \\
	&= \sum_{\alpha,\beta\in\sigma(\phi)} [\log(|\alpha|)v_\alpha,w_\beta] + [v_\alpha,\log(|\beta|)w_\beta] \\
	&= [\tilde A v,w] + [v,\tilde A w] .
	\end{align*}
	Therefore $\tilde A \in\Der(\frk g)$, as claimed.
	
	Notice that $N=\Id + \psi$ with $\psi$ nilpotent linear map on $\frk g_\C$.
	Indeed, if $v_\alpha\in E^\phi_\alpha$, then there is $m\in\N$ such that 
	\[
	\alpha^m (N-\Id)^mv_\alpha
	= \alpha^m (\phi_n\circ\phi-\Id)^mv_\alpha
	= (\phi-\alpha\Id)^mv_\alpha
	= 0 .
	\]
	Since $\alpha\neq0$ because $\phi$ is injective, then $(\phi_n\circ\phi-\Id)^mv_\alpha=0$. 
	Since the number of non-trivial generalized eigenspaces of $\phi$ is finite, there is $m\in\N$ with $(N-\Id)^m=0$.
	
	Since $\psi$ is nilpotent, then
	\[
	D :=\log (N)= \log(\Id+\psi) = \sum_{k=1}^m \frac{(-1)^{k+1}}{k} \psi^k 
	\]
	is well defined and nilpotent, with $e^D=N\in\Aut_\C(\frk g_\C)$.
	By Lemma~\ref{lem10271051}, $D\in\Der_\C(\frk g_\C)$.

	Since $N(\frk g)=\frk g$, then $\psi(\frk g)\subset\frk g$ and thus $D(\frk g)\subset \frk g$.
	Therefore, $D\in\Der(\frk g)$.
	Since $N(E^\phi_\alpha)\subset E^\phi_\alpha$, then $\psi(E^\phi_\alpha)\subset E^\phi_\alpha$ and therefore $D(E^\phi_\alpha)\subset E^\phi_\alpha$.
	Since $\tilde A $ is diagonal on each generalized eigenspace, then $[\tilde A ,D]=0$.
	
	Finally, notice that $\phi=KRN$ and that 
	\[
	RN = e^{\tilde A }e^D 
	= e^{\tilde A +D} .
	\]
	Since $D$ is nilpotent, $\tilde A $ is diagonalizable, and $[\tilde A ,D]=0$, then $\sigma(\tilde A +D)=\sigma(\tilde A )\subset\R$.
	Finally, on the one hand $[\tilde A ,K]= [\tilde A ,\phi_k] = 0$; 
	On the other hand, $[D,K]=0$ because of $0=[N,K]=[\Id+\psi,K]=[\psi,K]$ and the formula defining $D$.
	So, the lemma is proven with $A=\tilde A +D$.
\end{proof}

\subsection{Proof of Theorem~\ref{thm12061648}}\label{sec12061913}
Theorem~\ref{thm12061648} follows from Lemma~\ref{lem10291010} and the following Lemma~\ref{lem12061951}.
\begin{lemma}\label{lem12061951}
	Let $(G,d,\delta,\lambda)$ be a self-similar metric Lie group.
	Then there are $K\in\Aut(\frk g)$ diagonalizable with $\sigma(K)\subset\bb S^1$, 
	$A\in\Der(\frk g)$ with $\sigma(A)\subset[1,\infty)$
	such that $[K,A]=0$, $\delta = K\lambda^A$ and there is an $A$-homogeneous distance on $G$ for which $\delta$ is still a dilation of factor $\lambda$.
\end{lemma}
\begin{proof}
	After Theorem~\ref{teo08221641}, we can identify $G$ and $\frk g$ via the exponential map.
	With this identification, $\delta=\delta_*$.
	Let $A\in\Der(\frk g)$ and $K\in\Aut(\frk g)$ as in Lemma~\ref{lem10271236} with $\phi=\delta$.
	Since $K$ is diagonalizable, $\sigma(K)\subset\bb S^1$ and $[A,K]=0$, then the closure $\scr K$ of the group generated by $K$ is a compact subgroup of $\Aut(\frk g)$ and $\delta\scr K\delta^{-1}=\scr K$. 
	By Lemma~\ref{lem10290953}, 
	there is an admissible distance $d'$ on $G$ such that $\delta$ is a dilation of factor $\lambda$ and $K$ is an isometry of $d'$.
	It follows that $\lambda^A$ is a dilation of factor $\lambda$ (remember that $\lambda$ is fixed).
	
	Define $d'':G\times G\to[0,+\infty]$ as
	\[
	d''(x,y) = \sup\left\{\frac{d'(\mu^Ax,\mu^Ay)}{\mu}\,:\, \mu>0 \right\}.
	\]
	We claim that, in fact,
	\begin{equation}\label{eq10291005}
	d''(x,y) = \max\left\{\frac{d'(\mu^Ax,\mu^Ay)}{\mu}\,:\, \mu\in[1,\lambda] \right\} .
	\end{equation}
	Indeed, if $\mu>0$ then there are $k\in\Z$ and $r\in[0,1]$ such that $\mu=\lambda^k\lambda^r$.
	Hence $\frac{d(\mu^Ax,\mu^Ay)}{\mu} = \frac{d'((\lambda^r)^Ax,(\lambda^r)^Ay)}{\lambda^r}$, where $\lambda^r\in[1,\lambda]$.
	Moreover, since $[1,\lambda]$ is compact and $\mu\mapsto \frac{d'(\mu^Ax,\mu^Ay)}{\mu}$ is continuous, the supremum is a maximum.
	
	We now claim that $d''$ is an $A$-homogeneous distance on $G$.
	It is clear that $d''$ is left-invariant and that, for every $\rho>0$ and $x,y\in G$, we have $d''(\rho^Ax,\rho^Ay)=\rho d''(x,y)$. 
	Moreover, from~\eqref{eq10291005} we get that $d''(x,y)<\infty$ and that $d''(x,y)>0$ whenever $x\neq y$.
	So, we are left to show the triangular inequality.
	Let $x,y,z\in G$.
	Then there is $\mu\in[1,\lambda]$ such that $d''(x,z) = \frac{d'(\mu^Ax,\mu^Az)}{\mu}$.
	We conclude that
	\[
	d''(x,z)  
		= \frac{d'(\mu^Ax,\mu^Az)}{\mu} 
		\le \frac{d'(\mu^Ax,\mu^Ay)}{\mu} + \frac{d'(\mu^Ay,\mu^Az)}{\mu} 
		\le d''(x,y) + d''(y,z) .
	\]
	Therefore, $d''$ is an $A$-homogeneous distance on $G$.
	Finally, since $[K,A]=0$, then $[K,\mu^A]=0$ for all $\mu>0$.
	Thus $K$ is still an isometry for $d''$, and $\delta=K\lambda^A$ is also a dilation of factor $\lambda$ for $d''$.
\end{proof}

\subsection{Proof of Theorem~\ref{thm12061651}}\label{sec12062041}
In the hypothesis of Theorem~\ref{thm12061651}.\ref{thm12061651item1}, $(G,d,\delta,\lambda)$ is a self-similar metric Lie group.
From Theorem~\ref{teo08221641} we get that $G$ is connected simply connected and the eigenvalues of $\delta_*$ have modulus smaller than or equal to $\lambda$ if $\lambda<1$, or greater than or equal to $\lambda$ if $\lambda>1$.

Let $A$ and $K$ as in Lemma~\ref{lem12061951} with $\delta_*=K\lambda^A$.
From Theorem~\ref{thm12041046} we get that $A$ is $\C$-diagonalizable on $V_1(A)$.
From Theorem~\ref{teo08221641}.\eqref{teo08221641item7} we also get that $V_1(A) = V_1(\lambda, \delta_*)$, and thus $\lambda^A$ is diagonalizable on $V_1(\lambda, \delta_*)$.
Since also $K$ is diagonalizable and $[K,e^A]=0$, then $\delta$ is also diagonalizable on $V_1(\lambda, \delta_*)$.

This shows that \ref{thm12061651item1} implies \ref{thm12061651item2} in Theorem~\ref{thm12061651}.

\medskip

Suppose now we are in the hypothesis of Theorem~\ref{thm12061651}.\ref{thm12061651item2}.
Let $K$ and $A$ as in Lemma~\ref{lem10271236} so that $\delta_*=K\lambda^A$, i.e., $\lambda^A=K^{-1}\delta_*$.
Since $[K,A]=0$, then $[K,\delta_*]=0$.
Therefore, since $K$ is diagonalizable and $\delta_*$ is diagonalizable on $V_1^A$, then $A$ is also diagonalizable on $V_1^A$.
From Theorem~\ref{thm12041046} we get that there is an $A$-homogeneous distance $d$ on $G$.
Since $K$ is diagonalizable, $\sigma(K)\subset\bb S^1$ and $[A,K]=0$, then the closure $\scr K$ of the group generated by $K$ is a compact subgroup of $\Aut(\frk g)$ and $\delta\scr K\delta^{-1}=\scr K$. 
Hence, by Lemma~\ref{lem10290953}, we can assume that $K$ is an isometry for $d$ and thus $\delta$ is also a dilation of factor $\lambda$ for $d$.

 
\subsection{Proof of Theorem~\ref{thm12061855}}\label{buon_natale}

If $(X,d)$ is a locally compact, isometrically homogeneous and homothetic metric space,
then it is connected by \cite[Proposition~3.7]{MR2865538}.
We apply  \cite{2017arXiv170509648C}, see also Theorem~\ref{thm12041706}, to obtain that
$(X,d)$ is isometric to a self-similar metric Lie group $(G,d,\delta,\lambda)$.
In particular, the space $(G,d)$ is a homothetic nilpotent metric Lie group, by Theorem~\ref{teo08221641}.
Then
Proposition~\ref{prop11071818} completes the existence statement of Theorem~\ref{thm12061855}.
The uniqueness of the group structure $G$ follows from \cite{MR3646026}, where it is proven that isometries of nilpotent Lie groups are Lie group isomorphisms.

\subsection{Reductions to real spectrum cases}
We finish off with two results that are one complementary to the other.
We first show that, when the spectrum of $A$ is in the line $1+i\R$, the only $A$-homogeneous metric spaces are Banach spaces.
In other words, any $A$-homogeneous distance is also $\Id$-homogeneous, where $\Id$ is the real diagonal of $A$.
We then show that, beyond this case, it is possible to find examples where such a reduction to the real spectrum is not possible.

\begin{proposition}\label{prop12182017}
Let $A$ be a derivation on the Lie algebra of a Lie group $G$ such that
	$V_1^A=\frk g$.
	Then $A$-homogeneous distances are vector norms.
\end{proposition}
\begin{proof}
	Let $d$ be a $A$-homogeneous distance on $G$.
	Then, $G$ is Abelian and simply connected by Theorem~\ref{thm12041046}, hence the exponential map $\exp:\frk g\to G$ is a Lie group isomorphism.
	We need to show that
	\begin{equation}\label{eq12182044}
	d(0,\lambda p) = \lambda d(0,p)
	\qquad\forall\lambda>0,\ p\in G.
	\end{equation}
	We fix a norm $\|\cdot\|$   on $\frk g$ and the corresponding operator norm on linear operators.
	
	By Theorem~\ref{thm12041046} again, $A$ is diagonalizable on the complex numbers.
	Hence, with the notation of Corollary~\ref{cor07251444}, $A_R=\Id$ and $A_N=0$.
	Hence, $\scr K=\overline{\{\lambda^{-1}\lambda^A\}_{\lambda>0}}$ is a compact subgroup of $\Aut(\frk g)$.
	It follows that there is a sequence $\lambda_k\to0^+$ such that
	\[
	\lim_{k\to\infty} \lambda_k^{-1}\lambda_k^A = \Id .
	\]
	
	Fix $p\in G$ and $\lambda>0$.
	Since $\lambda_k$ is an infinitesimal positive sequence, for every $\epsilon>0$ there is a function 
	\[
	\rho_\epsilon:\N\to\{\lambda_k:\|\lambda_k^{-1}\lambda_k^A - \Id\|<\epsilon\}
	\]
	such that $\lambda=\sum_{j\in\N}\rho_\epsilon(j)$.
	Define $p_\epsilon := \sum_{j\in\N} \rho_\epsilon(j)^A p$.
	Notice the following two facts:
	First,
	\[
	d(0,p_\epsilon)
	\le \sum_{j\in\N} d(0, \rho_\epsilon(j)^A p )
	= \sum_{j\in\N} \rho_\epsilon(j) d(0,p)
	= \lambda d(0,p) .
	\]
	Second,
	\begin{align*}
	\| p_\epsilon-\lambda p \|
	&= \left\| \sum_{j\in\N} \rho_\epsilon(j)^A p - \sum_{j\in\N} \rho_\epsilon(j) p \right\| 
	\le \sum_{j\in\N} \left\|   \rho_\epsilon(j)^A - \rho_\epsilon(j)\Id \right\| \|p\| \\
	&= \sum_{j\in\N} \rho_\epsilon(j) \left\|   \rho_\epsilon(j)^{-1}\rho_\epsilon(j)^A - \Id \right\| \|p\|
	\le \lambda\epsilon \|p\| .
	\end{align*}
	Therefore, $q_\epsilon\to \lambda p$ as $\epsilon\to0$ and, by the continuity of $d$,
	\begin{equation}\label{eq12171102}
	d(0,\lambda p) \le \lambda d(0,p) .
	\end{equation}
	Finally, since \eqref{eq12171102} holds for arbitrary $\lambda>0$ and $p\in\R^n$, we have also
	\[
	d(0,\lambda p) 
	\le \lambda d(0,p) 
	= \lambda d(0,\lambda^{-1}(\lambda p))
	\le d(0,\lambda p) .
	\]
	This shows~\eqref{eq12182044} and thus completes the proof.
\end{proof}

\begin{proposition}\label{prop12271522}
	There is a 
	locally compact, isometrically homogeneous and homothetic metric space
	that is not isometric to any $A$-homogeneous distance for $A$ with real spectrum.
\end{proposition}
\begin{proof}
	Let $(X,d)$ be the metric space described in Example~\ref{sec12271414}.
	Recall that $d$ is an admissible left-invariant distance on $X=\R^2$.
	
	Let $P_0$ and $I_0$ be the Lie groups of dilations and isometries, respectively, of $(\R^2,d)$ fixing $(0,0)$, and let $\frk p_0$ and $\frk i_0$ be their Lie algebras.
	Since $\R^2$ is nilpotent, we have $\frk i_0\subset \frk p_0\subset\frk{gl}(2)$.
	
	Recall that $d$ is $A$-homogeneous with $A:=\begin{pmatrix}2&-1\\1&2\end{pmatrix}$, that is $A\in\frk p_0$.
	The spectrum of $A$ is $\{2+i\}$.
	Suppose that $d$ were also $A'$-homogeneous for some $A'$ with real spectrum, i.e., $A'\in\frk p_0$.
	Therefore, since $A$ and $A'$ would be linearly independent and since, by Proposition~\ref{prop11071818}, we have $\dim(\frk p_0)=\dim(\frk i_0)+1$,
	then $\dim(\frk i_0)>0$, i.e., there would be $J\in\frk{gl}(\R^2)\setminus\{0\}$ such that $t\mapsto e^{tJ}$ were a one-parameter group of isometries of $(\R^2,d)$ fixing the origin $(0,0)$.
	
	Now, if $B$ is the unit ball of $d$ with center $(0,0)$, as we defined it in Example~\ref{sec12271414}, then $e^{tJ}B=B$ for all $t\in\R$.
	However, the only one-parameter subgroup of $\GL(\R^2)$ that fixes $B$ is the trivial group $\{\Id\}$.
	Thus $J=0$, which is a contradiction.
\end{proof}

\subsection{Proof of Theorem~\ref{thm05011958}}
Let $X$ be a metric space with a doubling measure~$\mu$ such that for $\mu$-a.e.~$p\in X$ there is a unique tangent $G_p$ to $X$ at $p$.
As it has been proved in \cite{MR2865538} (see especially Section~3.4 therein), for $\mu$-a.e.~$p\in X$, the space $G_p$ is 
 a locally compact, isometrically homogeneous, and homothetic metric space.
The conclusion follows from Theorem~\ref{thm12061855}.




\end{document}